\documentclass[a4paper,12pt]{amsart}
\usepackage{amssymb,amsmath,amsfonts,amsthm}
\usepackage[all]{xy}
\usepackage{enumerate}
\usepackage{mathrsfs}
\usepackage{graphicx}
\usepackage{thm-restate}

\newcommand{\R}{\mathbb{R}}

\newcommand{\C}{\mathbb{C}}
\newcommand{\N}{\mathbb{N}}

\newcommand{\m}{\mathfrak{m}}

\newcommand{\BB}{\mathcal{B}}
\newcommand{\Id}{\operatorname{Id}}

\newcommand{\Ker}{{\operatorname{Ker}}}

\newcommand{\Hess}{{\operatorname{Hess}}}

\newcommand{\ord}{{\operatorname{ord}}}
\newcommand{\Ap}{{\operatorname{Ap}}}
\newcommand{\gr}{{\operatorname{gr}}}
\newcommand{\Ann}{{\operatorname{Ann}}}

\newcommand{\pmt}[1]{\begin{pmatrix}#1\end{pmatrix}}
\newcommand{\bmt}[1]{\begin{bmatrix}#1\end{bmatrix}}

\def\aa{{\bf a}}

\newtheorem{pro}{Proposition}[section]
\newtheorem{Lem}[pro]{Lemma}
 
\newtheorem{Theo}[pro]{Theorem}

\newtheorem*{Theoetoile}{Theorem} 
\theoremstyle{definition}
\newtheorem{Defi}[pro]{Definition}

\numberwithin{equation}{section}
\let\epsilon\varepsilon
\let\kappa=\varkappa

\usepackage{hyperref}
\hypersetup{
	colorlinks=true,       
	linkcolor=red,          
	citecolor=blue,        
	filecolor=magenta,      
	urlcolor=cyan           
}

\title{{\footnotesize  The weak Lefschetz property for Artinian Gorenstein algebras of codimension three}}
\date{\today}

\author{Rosa M. Mir\'o-Roig}
\address{Universitat de Barcelona, Departament de Matem\`atiques i Inform\`atica, Gran Via de les Corts Catalanes 585, 08007 Barcelona, Spain.}
\email{miro@ub.edu}

\author{Quang Hoa Tran}
\address{University of Education, Hue University,  34 Le Loi St., Hue City, Vietnam.}
\email{tranquanghoa@hueuni.edu.vn}

\begin{document}
\maketitle
\begin{abstract}
We study the weak Lefschetz property of  a class of graded Artinian Gorenstein algebras of codimension three associated in a natural way to the Ap\'ery set of a numerical semigroup generated by four natural numbers. We show that  these algebras have the weak Lefschetz property whenever the initial degree  of their defining ideal  is small.\\
{{\footnotesize  \textsc{Keywords}:   Ap\'ery set,  Artinian Gorenstein algebras,  Hessians, Macaulay dual generators, numerical semigroups, weak Lefschetz property.}}\\
\texttt{MSC2010}: primary 13E10, 13H10; secondary 13A30, 13C40.
\end{abstract}

\section{Introduction}
The weak Lefschetz property (WLP for short) for an Artinian graded algebra $A$ over a field $K$ simply says that there exists a linear form $L$ that induces, for each $i$, a multiplication map $\times L: [A]_{i}\longrightarrow [A]_{i+1}$ that has maximal rank, i.e. that is either injective or surjective. At first glance this might seem to be a simple problem of linear algebra. However, determining which graded Artinian $K$-algebras have the WLP is notoriously difficult. Many authors have studied the problem from many different points of view, applying tools from representation theory, topology, vector bundle theory, plane partitions, splines, differential geometry, among others (see for instance \cite{BK2007, HSS2011,  Watanabe2013, MMO2013,  MRT19, Stanley80}). The role of the characteristic of $K$ in this problem has also been an important, and only superficially understood, aspect of these studies.

One of the most interesting open problems in this field is whether all codimension 3 graded Artinian Gorenstein algebras have the WLP in characteristic zero. In the special case of codimension 3 complete intersections, a positive answer was obtained in characteristic zero in \cite{HMNW2003} using the Grauert-M\"{u}lich theorem. For positive characteristic, on the other hand, only the case of monomial complete intersections has been studied (see \cite{BK2011,CookII2012,CN2011}), applying many different approaches from combinatorics.

For the case of codimension 3 Gorenstein algebras that are not necessarily complete intersections, it is known that for each possible Hilbert function an example exists having the WLP \cite{Harima1995}. Some partial results are given in \cite{MZ2008} to show that for certain Hilbert functions, all such Gorenstion algebras have the WLP. It was shown in \cite{BMMNZ14} that  all codimension 3 Artinian Gorenstein algebras of socle degree at most 6 have the WLP in characteristic zero.  But the general case remains completely open.

In this work, we consider a class of graded Artinian Gorenstein algebras of codimension 3 built up starting from the  Ap\'ery set of a numerical semigroup generated by 4 natural numbers. Our goal is to study whether these algebras have the WLP.  More precisely,  we consider a numerical semigroup $P$ generated by  $\{a_1,a_2,a_3,a_4\}\subset \N^4$ such that $\gcd(a_1,a_2,a_3,a_4)=1$.  The \textit{Ap\'ery set $\Ap(P)$ of $P$} with respect to the minimal generator of the semigroup is defined as follows
$$\Ap(P):=\{a\in P\mid a-a_1\notin P\}=\{0=\omega_1<\omega_2<\cdots<\omega_{a_1}\}.$$
Notice that $\Ap(P)$ is a finite set and $\# \Ap(P)=a_1.$  Recall that a numerical semigroup $P$ is said to be \textit{$M$-pure symmetric} if for each $i=1,\ldots,a_1$, $\omega_i+\omega_{a_1-i+1}=\omega_{a_1}$ and $\ord(\omega_i)+\ord(\omega_{a_1-i+1})=\ord(\omega_{a_1})$, where 
$$ \ord(a):=\max \{ \sum_{i=1}^{4} \lambda_i \mid  a= \sum_{i=1}^{4} \lambda_ia_i  \}$$
is the \textit{order} of $a\in P$. Therefore the Ap\'ery set of a $M$-pure symmetric semigroup has the structure of a symmetric lattice.

Let $K$ be a field of characteristic zero and consider the homomorphism 
$$\Phi: S:=K[x_1,\ldots,x_4]\longrightarrow K[P]:=K[t^{a_1},\ldots,t^{a_4}],$$
which sends $x_i\longmapsto t^{a_i}.$ Then $K[P]\cong S/\Ker(\Phi)$ is a one dimensional ring associated to $P$. Now set $\overline{S} =S/(x_1).$ Then there is one to one correspondence between the elements of $\Ap(P)$ and the generators of $\overline{S}$ as a $K$-vector space. Let $\overline{\m}$ be the maximal homogeneous ideal of $\overline{S}$, define the \textit{associated graded algebra} of the Ap\'ery set of $P$
$$A=\gr_{\overline{\m}}(\overline{S}):=\bigoplus_{i\geq0}\frac{\overline{\m}^i}{\overline{\m}^{i+1}}.$$
It follows that $A$ is a standard graded Artinian $K$-algebra. In the work \cite{Bryant2010}, Bryant proved that $A$ is Gorenstein if and only if $P$ is $M$-pure symmetric. In \cite{Gu2018}, Guerrieri showed that if $A$ is an Artinian Gorenstein algebra that is not a complete intersection, then $A$ is of form $A=R/I$ with $R=K[x,y,z]$ and  
\begin{align} \label{idealGorenstein}
I=(x^a, y^b-x^\alpha z^\gamma, z^c, x^{a-\alpha}y^{b-\beta}, y^{b-\beta}z^{c-\gamma})\subset R,
\end{align}
where $1\leq \alpha\leq a-1,\; 1\leq \beta\leq b-1,\; 1\leq \gamma\leq c-1$ and $\alpha+\gamma =b.$ The integers $a,b,c,\alpha,\beta$ and $\gamma$ are determined by the structure of $\Ap(P)$, see \cite[Section 5]{Gu2018}.

Our goal is to study the WLP for $A$. Our main result is the following (see Theorems~\ref{Theorem3.6} and~\ref{Theorem3.13}).
\begin{Theoetoile}
Consider the ideal $I$ as in \eqref{idealGorenstein}. If one of the integers $a,b$ and $c$ is less than or equal to three, then $R/I$ has the WLP.
\end{Theoetoile}

\section{Artinian Gorenstein algebras}
In this section, we will recall some standard notations and known facts that will be needed later in this work. We fix $K$ a field of characteristic zero and $R=K[x_1,\ldots,x_n]$ a standard graded homogeneous polynomial ring in $n$ variables over $K$. Let
$$A=R/I=\bigoplus_{i=0}^D [A]_i$$
be a graded Artinian algebra. Note that $A$ is finite dimensional over $K.$

\begin{Defi}
For any graded Artinian  algebra $A=R/I=\bigoplus_{i=0}^D [A]_i$, the \textit{Hilbert function} of $A$ is the function
$$h_A: \N\longrightarrow \N$$
defined by $h_A(t)=\dim_K [A]_t$. As $A$ is Artinian, its Hilbert function is equal to its \textit{$h$-vector} that one can express as a sequence
$$\underline{h}_A=(1=h_0,h_1,h_2,h_3,\ldots, h_D),$$
with $h_i=h_A(i)>0$ and $D$ is the last index with this property. The integer $D$ is called the \textit{socle degree} of $A$. The $h$-vector  $\underline{h}_A$ is said to be \textit{symmetric} if $h_{D-i}=h_i$ for every $i=0,1,\ldots,\lfloor\frac{D}{2}\rfloor.$ 
\end{Defi}

\begin{Defi} \cite[Proposition~2.1]{MW2009}
A standard graded Artinian algebra $A$ as above is Gorenstein if and only if $h_D=1$ and the multiplication map 
$$[A]_i\times [A]_{D-i}\longrightarrow [A]_D\cong K$$ is a perfect pairing for all $i=0,1,\ldots,\lfloor\frac{D}{2}\rfloor.$
\end{Defi}
It follows that the $h$-vector of a graded Artinian Gorenstein is symmetric.

\begin{Defi}
A graded Artinian $K$-algebra $A$ is said to have the \textit{weak Lefschetz property}, briefly WLP, if there exists an element $L\in [A]_1$ such that the multiplication map $\times L: [A]_i\longrightarrow [A]_{i+1}$ has maximal rank for each $i$. We also say that a homogeneous ideal $I$ has the WLP if $R/I$ has the WLP.
\end{Defi}

From now on, we only consider a standard graded Artinian Gorenstein $K$-algebra. For these algebras, the WLP is determined by considering only the multiplication map in one degree. 
\begin{pro}\cite[Proposition~2.1]{MMN2011}\label{Proposition2.5}
Let $A$ be a standard graded Artinian Gorenstein $K$-algebra with the socle degree $D$ and $k:=\lfloor\frac{D}{2}\rfloor.$ Then we have: 
\begin{enumerate}
\item[\rm (i)]  If $D$ is odd, $A$ has the WLP if and only if there is an element $L\in [A]_1$ such that the multiplication map $\times L: [A]_k\longrightarrow[A] _{k+1}$ is an isomorphism. 
\item[\rm (ii)]  If $D$ is even, $A$ has the WLP if and only if there is an element $L\in [A]_1$ such that the multiplication map $\times L: [A]_k\longrightarrow[A] _{k+1}$ is surjective or equivalently the  multiplication map $\times L: [A]_{k-1}\longrightarrow[A] _{k}$ is injective. 
\end{enumerate}
\end{pro}
\begin{pro}\cite[Theorem 2.1]{Gu2018}\label{Proposition2.7}
Assume that $G=\bigoplus_{i=0}^D [G]_i$ is a standard graded Artinian Gorenstein $K$-algebra with the socle degree $D$ that has the WLP. If $\ell\in [G]_1$ is a linear element, then the quotient ring
$$A=\frac{G}{(0\colon_G \ell)}$$
is also a standard graded Artinian Gorenstein $K$-algebra. Assume that $G$ and $A$ have the same codimension and set $k:=\lfloor\frac{D}{2}\rfloor.$ Then
\begin{enumerate}
\item [\rm (i)] If  $D$  is odd, then $A$ has the WLP.
\item [\rm (ii)] If $D$  is even and $\dim_K [G]_{k-1}=\dim_K [G]_{k}$, then  $A$  has the WLP.
\end{enumerate}
\end{pro}
An important tool needed to study whether a Gorenstein algebra has the WLP is the Macaulay inverse system, and especially  the higher Hessians. We give now some definitions and results taken from a paper by Maeno and Watanabe \cite{MW2009} and from a recent paper by Gondim and Zappal\'a \cite{GZ2018}. The general facts on the Macaulay's inverse system can be seen in \cite{Eil18}.

Now we regard $R$ as an $R$-module via the operation ``$\circ$'' defined by
\begin{align*}
&R\times R \longrightarrow R\\
& (x^\alpha,x^\beta) \longmapsto x^\alpha \circ x^\beta =\begin{cases}
x^{\beta -\alpha}&\text{if}\quad \beta_i \geq \alpha_i,\forall i=1,\ldots,n\\
0 & \text{otherwise}
\end{cases}
\end{align*}
with $x^\alpha =x_1^{\alpha_1}\cdots x_n^{\alpha_n}$ and $x^\beta =x_1^{\beta_1}\cdots x_n^{\beta_n}$. For a polynomial $F\in R$, $\Ann_R(F)$ denotes
$$\Ann_R(F):=\{f\in R\mid f\circ F =0  \} $$
which is an ideal of $R$. It is called  the \textit{annihilator} of $F$. It is known that $R/\Ann_R(F)$ is an Artinian Gorenstein algebra. Furthermore, every Artinian Gorenstein algebra can be written in this form. More precisely, we have the following.
\begin{pro}\cite[Theorem 2.1]{MW2009}
Let $I$ be an ideal of $R$ and $A =R/I$ the quotient algebra. Denote by $\m$ the homogeneous maximal ideal of $R$. Then
$\sqrt{I}=\m$ and the $K$-algebra $A$ is Gorenstein if and only if there exists a polynomial $F\in R$  such that $I=\Ann_R(F)$.
\end{pro}
The polynomial $F$ in the above proposition is called the \textit{Macaulay dual generator} of $A=R/\Ann_R(F).$ Furthermore, if $F$ is a homogeneous polynomial of degree $D$, then $R/\Ann_R(F)$ is a graded Artinian Gorenstein algebra of socle degree $D.$
\begin{Defi}
Let $F$ be a polynomial in $R$ and $d,k\geq 1$  be two integers. Assume that $\mathcal{B}_d=\{ \alpha_i\}_{i=1}^s$ and  $\mathcal{B}_k=\{ \beta_j\}_{=1}^t$ form respectively the $K$-linear basis of $[A]_d$ and $[A]_k$. We define the mixed Hessian of $F$ as an $(s\times t)$-matrix
$$\Hess^{d,k}_{\mathcal{B}_d,\mathcal{B}_k} (F):= \pmt{(\alpha_i\cdot \beta_j)\circ F}.$$
In particular, if $d=k$, then we define the $d$-th Hessian of $F$ as a square matrix
$$\Hess^{d}_{\mathcal{B}_d} (F):= \pmt{(\alpha_i\cdot \alpha_j)\circ F}.$$
\end{Defi}
Notice that the singularity of these matrices  is independent of the chosen basis and hence we can write simply $\Hess^{d} (F)$ and $\Hess^{d,k} (F)$. Based on the singularity of (mixed) Hessians of $F$, we can determine  the WLP of $A=R/\Ann_R(F)$.
\begin{pro}\cite{GZ2018} \label{Proposition2.8_Hessian}
Assume that $A=R/\Ann_R(F)$ with $F\in [R]_D$ and $k:=\lfloor\frac{D}{2}\rfloor$. Then we have:
\begin{enumerate}
\item [\rm (i)] If $D$ is odd, then $A$ has the WLP if and only if the Hessian $\Hess^{k} (F)$ has maximal rank, i.e., it has nonzero determinant.
\item [\rm (ii)] If $D$ is even, then  $A$ has the WLP if and only if the mixed Hessian $\Hess^{k-1,k} (F)$ has maximal rank.
\end{enumerate}
\end{pro}
We close this section by recalling a result on the WLP of codimension 3 Artinian Gorenstein algebras.
\begin{pro}\cite[Corollary~3.12]{BMMNZ14}\label{Proposition2.6}
In characteristic zero, all codimension 3 Artinian Gorenstein algebras of socle degree at most 6 have the WLP.
\end{pro}

\section{The WLP for class of Artinian Gorenstein algebras of codimension 3}

From now on, let $R=K[x,y,z]$ be the standard graded polynomial ring over a field $K$ of characteristic zero and consider the ideal 
\begin{align}\label{IdealGorenstein}
I=(x^a, y^b-x^\alpha z^\gamma, z^c, x^{a-\alpha}y^{b-\beta}, y^{b-\beta}z^{c-\gamma})\subset R,
\end{align}
where $1\leq \alpha\leq a-1,\, 1\leq \beta\leq b-1$ and $1\leq \gamma\leq c-1$ such that $\alpha+\gamma =b.$ It is clear that $b\leq a+c-2$ and by symmetry of $x$ and $z$, without loss of generality, we assume that $a\geq c$. 
First, we have the following.

\begin{pro} \label{Proposition3.1}
 Fix $a,b,c,\alpha,\beta,\gamma$ as above. Set $\aa= (x^a, y^b-x^\alpha z^\gamma, z^c)$. Then one has:
\begin{enumerate}
\item[\rm (i)] $I=\aa \colon_R y^\beta.$ Therefore, $R/I$ is an Artinian Gorenstein of codimension 3 and the socle degree of $R/I$ is $D=a+b+c-\beta-3.$
\item[\rm (ii)]  The Macaulay dual generator of $R/I$ is
$$F=\sum_{i=0}^{m}x^{a-1-i\alpha}y^{(i+1)b-1-\beta}z^{c-1-i\gamma},$$
where $m:=\max\{j\mid a-1-j\alpha\geq 0\; \text{and}\; c-1-j\gamma\geq 0\}.$
\item[\rm (iii)]  The free resolution of $R/I$ is
\begin{align*}
{\footnotesize  \xymatrix{ 0\longrightarrow R(-a-b-c+\beta) \ar[r] & {\begin{array}{c} R(-a-b+\beta)\\ \oplus \\R(-a-c+\beta)\\\oplus \\R(-b-c+\beta)\\\oplus \\R(-a-\gamma)\\ \oplus \\R(-c-\alpha) \end{array}}\ar[r]^M &  {\begin{array}{c} R(-a)\\ \oplus \\R(-b)\\\oplus \\R(-c)\\\oplus \\R(-a-\gamma+\beta)\\ \oplus \\R(-c-\alpha+\beta ) \end{array}}\ar[r] & R \ar[r]&R/I\longrightarrow 0,}}
\end{align*}
where $M$ is a skew-symmetric matrix
$$M=\bmt{0& y^{b-\beta} & 0 & -x^\alpha & 0 \\ -y^{b-\beta} &0 & z^\gamma & 0 & 0 \\ 0& -z^\gamma & 0 & y^\beta & -x^{a-\alpha}\\ x^\alpha &0 & -y^\beta & 0 & z^{c-\gamma}\\ 0& 0 & x^{a-\alpha}& -z^{c-\gamma} & 0}.$$
\end{enumerate}
\end{pro}
\begin{proof}
Firstly, since $\aa$ is a complete intersection, $I=\aa\colon y^\beta$ is Gorenstein. This proves (i). It is known that $R/\Ann_R(F)$ is an Artinian Gorenstein algebra of socle degree $a+b+c-\beta-3$. Since $I\subset \Ann_R(F)$ and $R/I$ is an Artinian Gorenstein algebra, by \cite[Lemma 1.1]{KU1992}, $I=\Ann_R(F)$. The item (ii) is proved. Finally, (iii) is implied from the structure theorem of Gorenstein ideals of codimension 3 and also from a standard mapping cone computation.
\end{proof}

One of the interesting open problems  is whether all codimension 3 graded Artinian Gorenstein algebras have the WLP in characteristic zero. Now let $I$ be an ideal as in \eqref{IdealGorenstein}.  By the above proposition, $R/I$ is a  graded Artinian Gorenstein algebra of codimension 3, hence we are interested in studying the WLP for $R/I$. In the next subsections, we will prove that $R/I$ has the WLP whenever the initial degree of $I$ is at most three. In the paper, we denote by $\Id$ the identity matrix and by $M^t$ the transpose matrix of a matrix $M$.
\subsection{The ideal  $I$ contains a quadric.}
In this subsection, we consider the simplest case where the ideal $I$ contains a quadric.

The first case is $b=2$, hence $\alpha=\beta=\gamma=1$. We obtain the following result.
  
\begin{pro} \label{Propositionb=2}
Let $I$ be the ideal
$$I=(x^a, y^2-xz, z^c, x^{a-1}y, yz^{c-1})\subset R$$
with $a\geq c\geq  2.$ Then $R/I$ has the WLP.
\end{pro}
\begin{proof} 
The socle degree of $R/I$ is $D=a+c-2$. Set
$k:=\lfloor\frac{D}{2}\rfloor =\lfloor\frac{a+c-2}{2}\rfloor.$ Set $L=x-y+z$. By Proposition~\ref{Proposition2.5}, it is enough to show that
$$\times L: [R/I]_k\longrightarrow [R/I]_{k+1}$$
is surjective, or equivalently $[R/(I,L)]_{k+1}=0.$ We have that
\begin{align*}
R/(I,L)\cong K[x,z]/J,
\end{align*}
where $J=(x^a,x^2+xz+z^2,z^c,x^{a-1}z, xz^{c-1})$. We will prove that $[K[x,z]/J]_{k+1}=0$, or equivalently $x^iz^{k+1-i}\in J$ for all $0\leq i\leq k+1.$ We do it by induction on $i$. As $a\geq c,$ hence $c\leq k+1$. It follows that $z^{k+1}$ and $xz^k$ belong to $J.$ For any $i\geq 2$, one has
$$x^iz^{k+1-i}=x^2x^{i-2}z^{k+1-i} =-(z^2+xz)x^{i-2}z^{k+1-i} =-x^{i-2}z^{k+3-i}-x^{i-1}z^{k+2-i}\in J,$$
by the induction hypothesis . 
\end{proof}

We now study the case $a=2$ or $c=2$. By symmetry of $x$ and $z$, WLOG, we can assume $a\geq c=2$. Therefore $\gamma=1$ and $a\geq \alpha+1=b.$ More precisely, we consider the ideal
$$I_\beta=(x^a, y^b-x^{b-1} z, z^2, x^{a-b+1}y^{b-\beta}, y^{b-\beta}z)\subset R,$$
with  $1\leq \beta \leq b-1$ and $a\geq b$. Set $A_\beta =R/I_\beta$. By Proposition~\ref{Proposition3.1}, the free resolution of $A_\beta$ is
\begin{align}\label{freeresolution}
 {\footnotesize \xymatrix{ 0\ar[r] & R(-a-b-2+\beta) \ar[r] & {\begin{array}{c} R(-a-2+\beta)\\\oplus \\R(-b-2+\beta) \\ \oplus \\R(-a-b+\beta)\\\oplus  \\R(-a-1)\\ \oplus \\R(-b-1)  \end{array}}\ar[r] &  {\begin{array}{c} R(-2)\\\oplus \\  R(-a)\\\oplus \\R(-b)\\ \oplus\\ R(-a-1+\beta) \\\oplus \\R(-b-1+\beta) \end{array}}\ar[r] & R\ar[r] & A_\beta \ar[r] &  0}.}
\end{align}
Since we have the free resolution \eqref{freeresolution} of $A_\beta$, for any integer $j\geq 2$, we get
\begin{align}\label{Congthu3.3}
\quad H_{A_\beta} (j)=& 2j+1 -\binom{j-a+1}{1}-\binom{j-b+1}{1}-\binom{j-a+\beta}{1}-\binom{j-b+\beta}{1}\\
&+\binom{j-a-b+\beta+1}{1}+\binom{j-a-b+\beta}{1}\nonumber,
\end{align}
with convention $\binom{n}{m}=0$ if $n<m.$ By Proposition~\ref{Proposition3.1}, the  socle degree of $A_\beta$ is $D=a+b-\beta-1$. Set $k:=\lfloor \frac{D}{2}\rfloor$. Then  $k-a<0$ and $k-a-b+\beta<0$, it follows from \eqref{Congthu3.3}  that
\begin{align} \label{Congthuchilbertfunction}
H_{A_\beta} (k)& = 2k+1 -\binom{k-b+1}{1}-\binom{k-a+\beta}{1}-\binom{k-b+\beta}{1}.
\end{align}
The Hilbert function of $A_\beta$ in degree $k$ is determined as follows.
\begin{Lem}\label{lemmakey}
For every $1\leq \beta\leq b-1$, one has
\begin{align*}
H_{A_\beta} (k)&=\begin{cases} 
2b-\beta&\text{if}\quad \beta \leq a-b\\
a+b-2\beta+1&\text{if}\quad \beta\geq a-b+1.
\end{cases}
\end{align*}
Furthermore, if $1\leq \beta \leq a-b-1$, then 
\begin{align*}
H_{A_\beta} (k) =H_{A_\beta} (k-1).
\end{align*}
\end{Lem}
\begin{proof}
Firstly, we consider the case where $a+b-\beta$ is even. Hence $k=\frac{a+b-\beta}{2}-1.$ It follows from \eqref{Congthuchilbertfunction} that 
\begin{align*} 
H_{A_\beta} (k)=a+b-\beta-1-\binom{\frac{a-b -\beta}{2}}{1}-\binom{\frac{a-b+\beta}{2}-1}{1}-\binom{\frac{b-a+\beta}{2}-1}{1}.
\end{align*}
Since $a\geq b\geq \beta +1\geq 2$. We consider the following cases.

\noindent{\bf \underline{Case 1: $a=b$}.} In this case,  $\beta$ has to be even and it is  easy to show  that 
$$H_{A_\beta} (k)=a+b-2\beta+1.$$

\noindent{\bf \underline{Case 2: $a=b+1$}.} In this case,  $\beta$  must be odd.  Therefore
\begin{align*} 
H_{A_\beta} (k)& = a+b-\beta-1-\binom{\frac{\beta+1}{2}-1}{1}-\binom{\frac{\beta-1}{2}-1}{1}\\
&=\begin{cases}
2b-\beta &\text{if}\quad \beta =1\\
a+b-2\beta +1 &\text{if}\quad \beta \geq 3.
\end{cases}
\end{align*}
\noindent{\bf \underline{Case 3: $a=b+2$}.} In this case,  $\beta$  must be even, hence $\beta\geq 2$.  It follows that
\begin{align*} 
H_{A_\beta} (k)& = a+b-\beta-1-\binom{\frac{\beta+2}{2}-1}{1}-\binom{\frac{\beta-2}{2}-1}{1}\\
&=\begin{cases}
2b-\beta &\text{if}\quad \beta =2\\
a+b-2\beta +1 &\text{if}\quad \beta \geq 4.
\end{cases}
\end{align*}

\noindent{\bf \underline{Case 4: $a\geq b+3$}.} Then $a-b+\beta \geq 4.$ Therefore, if $\beta\geq a-b+4$, then
\begin{align*}
H_{A_\beta} (k)&=a+b-\beta-1 -\frac{a-b+\beta}{2} +1- \frac{b-a+\beta}{2}+1 \\
&=a+b-2\beta+1.
\end{align*}
If $\beta \leq a-b-2$, then 
\begin{align*}
H_{A_\beta} (k)&=a+b-\beta-1 -\frac{a-b -\beta}{2}-\frac{a-b+\beta}{2}+1\\
&=2b-\beta. 
\end{align*}
Thus we only consider the case $\beta=a-b+i,\; -1\leq i\leq 3.$ But $a+b-\beta$ is even, therefore both $\beta$ and $a-b$ are either even or odd. It follows that we only consider the two cases where  $\beta =a-b+2$ or $\beta= a-b$. If $\beta =a-b+2$, then a straightforward computation shows that 
\begin{align*}
H_{A_\beta} (k)=a+b-2\beta+1.
\end{align*}

Similarly, if $\beta= a-b$ then 
$$H_{A_\beta} (k) =3b-a=2b-\beta.$$ 

Thus we conclude that
\begin{align*}
H_{A_\beta} (k)&=\begin{cases} 
2b-\beta&\text{if}\quad \beta \leq a-b\\
a+b-2\beta+1&\text{if}\quad \beta\geq a-b+2
\end{cases}
\end{align*}
as desired.

Secondly, we consider the case where $a+b-\beta$ is odd. Hence $k=\frac{a+b-\beta-1}{2}.$ It follows from \eqref{Congthuchilbertfunction} that 
\begin{align*} 
H_{A_\beta} (k)=a+b-\beta-\binom{\frac{a-b -\beta+1}{2}}{1}-\binom{\frac{a-b+\beta-1}{2}}{1}-\binom{\frac{b-a+\beta-1}{2}}{1}.
\end{align*}
Since $a\geq b\geq \beta +1\geq 2$. We consider the following cases where $a=b,\ a=b+1$ or $a\geq b+2.$ The proof is similar as above (even more simple).

Finally, if $1\leq \beta \leq a-b-1$, then 
\begin{align*}
H_{A_\beta} (k) =2b-\beta.
\end{align*}
Notice that $k-a+\beta< 0$ since $a-b\geq \beta+1$.  It follows from \eqref{Congthu3.3} that 
\begin{align*}
H_{A_\beta} (k-1)& = 2k-1 -\binom{k-b}{1}-\binom{k-b+\beta-1}{1}.
\end{align*}
If  $a+b+\beta$ is odd, then 
\begin{align*}
H_{A_\beta} (k-1)&=a+b-\beta-2-\binom{\frac{a-b -\beta-1}{2}}{1}-\binom{\frac{a-b+\beta-1}{2}-1}{1}\\
&= \begin{cases}
2b-\beta&\text{if}\quad a-b=\beta+1\\
2b-\beta&\text{if}\quad a-b=\beta+3\\
2b-\beta&\text{if}\quad a-b\geq \beta+5
\end{cases}\\
&=2b-\beta.                           
\end{align*}
If  $a+b+\beta$ is even, then 
\begin{align*}
H_{A_\beta} (k-1)&=a+b-\beta-3-\binom{\frac{a-b -\beta}{2}-1}{1}-\binom{\frac{a-b+\beta}{2}-2}{1}\\
&= \begin{cases}
2b-\beta&\text{if}\quad a-b=\beta+2\\
2b-\beta&\text{if}\quad a-b\geq \beta+4
\end{cases}\\
&=2b-\beta.                           
\end{align*}
Thus the lemma is completely proved.
\end{proof}

\begin{Lem}\label{LemmahilbertfunctionforG}
Set $G= R/(x^a, y^b-x^{b-1} z, z^2)$ and $k=\lfloor\frac{a+b-1}{2}\rfloor$. If $a\geq b,$ then
\begin{align*}
H_{G} (k)=\begin{cases}
2b&\text{if}\quad a\geq b+1\\
2b-1&\text{if}\quad a=b.
\end{cases}
\end{align*}
Furthermore, if $a\geq b+3$, then 
\begin{align*}
H_{G} (k) =H_{G} (k-1).
\end{align*}
\end{Lem}
\begin{proof}
Since $G$ is resolved by the Koszul complex and $k-a<0$, we have
\begin{align*}
H_G(k)&=\binom{k+2}{2}-\binom{k}{2}-\binom{k-b+2}{2}-\binom{k-b}{2}\\
&=2k+1-\binom{k-b+1}{1}-\binom{k-b}{1}.
\end{align*}
If $a+b$ is odd, then $k=\frac{a+b-1}{2}.$ A simple computation shows that
\begin{align*}
H_G(k)&=\begin{cases}
a+b-\frac{a-b-1}{2}-1-\frac{a-b-1}{2}&\text{if}\quad a-b\geq 3\\
a+b-1&\text{if}\quad a-b=1
\end{cases}\\
&=2b.
\end{align*}
If $a+b$ is even, then $k=\frac{a+b}{2}-1.$ It follows that
\begin{align*}
H_G(k)&=\begin{cases}
a+b-1-\frac{a-b}{2}-\frac{a-b}{2}+1&\text{if}\quad a-b\geq 4\\
a+b-1-1&\text{if}\quad a-b= 2\\
a+b-1&\text{if}\quad a=b
\end{cases}\\
&=\begin{cases}
2b&\text{if}\quad a-b\geq 2\\
2b-1&\text{if}\quad a=b.
\end{cases}
\end{align*}
Analogously we can check that
\begin{align*}
H_G(k-1)&=2k-1-\binom{k-b}{1}-\binom{k-b-1}{1}\\
&=\begin{cases}
2b&\text{if}\quad a-b\geq 3\\
2b-1&\text{if}\quad 1\leq a-b\leq 2\\
2b-3&\text{if}\quad  a=b.
\end{cases}
\end{align*}
Thus, if $a-b\geq 3$, then $H_{G} (k) =H_{G} (k-1).$
\end{proof}

\begin{pro} \label{Propositionforbetalarge}
Assume  $b\leq a\leq 2b-3$. Then the ideal
$$I_\beta=(x^a, y^b-x^{b-1} z, z^2, x^{a-b+1}y^{b-\beta}, y^{b-\beta}z)\subset R$$
has the WLP, whenever $a-b+2\leq \beta \leq b-1$.
\end{pro}
\begin{proof}
Set $A_\beta =R/I_\beta$. By Proposition~\ref{Proposition3.1}, one has
$$ \quad A_\beta =\frac{A_{\beta -1}}{(0\colon_{A_{\beta-1}} y)},$$
for all $2\leq \beta\leq b-1$.
Notice that the socle degree of $A_\beta$ is $D=a+b-\beta-1.$ Hence if $\beta=a-b+2$, then $ D=2b-3$ is odd. To prove that $A_\beta$ has the WLP for every $a-b+2\leq\beta \leq b-1$, by Proposition~\ref{Proposition2.7}(i), it is enough to prove that $A_\beta$ has the WLP whenever $D$ is odd.

Now let $\beta$ be an integer such that $a-b+2\leq\beta \leq b-1$ and $a+b-\beta$ is even. In this case, one has $k =\frac{a+b-\beta}{2}-1$. It follows from Lemma~\ref{lemmakey} that
\begin{align*}
H_{A_\beta} (k)= a+b-2\beta+1.
\end{align*}
Clearly,  since $\beta \geq a-b+2$,  $k<b\leq a$. Therefore, we can take  a $K$-linear basis $\mathcal{B} =\mathcal{B}_1\sqcup \mathcal{B}_2\sqcup \mathcal{B}_3$ of $[A_\beta]_{k}$ with
\begin{align*}
\mathcal{B}_1 &= \{ u_i =x^{k+1-i} y^{i-1} \mid i=1,2,\ldots,b-\beta \}\\
\mathcal{B}_2 &= \{ v_i =x^{i-1} y^{k+1-i} \mid i=1,2,\ldots,a-b+1 \}\\
\mathcal{B}_3 &= \{ w_{i} =x^{k-i} y^{i-1} z\mid i=1,2,\ldots, b-\beta  \}.
\end{align*}
On the other hand, the Macaulay dual generator of $A_\beta$ is
$$F=x^{a-1}y^{b-\beta -1}z +x^{a-b}y^{2b-\beta -1}.$$
To prove the proposition, by Proposition~\ref{Proposition2.8_Hessian}, it is enough to show that $\Hess^{k}_\mathcal{B}(F)$
has nonzero determinant. Write 
$$\Hess^{k}_\mathcal{B}(F)=\bmt{ A &\vdots & B&\vdots & C\\\cdots & \cdots &\cdots&\cdots&\cdots\\ B^t & \vdots & U& \vdots & V\\\cdots & \cdots &\cdots&\cdots&\cdots\\ C^t & \vdots & V^t& \vdots & W},$$
where $A=\pmt{(u_i\cdot u_j)\circ F}, B=\pmt{(u_i\cdot v_j)\circ F}, C=\pmt{(u_i\cdot w_j)\circ F}, U=\pmt{(v_i\cdot v_j)\circ F},$
$V=\pmt{(v_i\cdot w_j)\circ F} $ and $W=\pmt{(w_i\cdot w_j)\circ F}$.

Notice that $A,C,U$ and $W$ are the square matrices. It follows that the diagonal of $\Hess^{k}_\mathcal{B}(F)$ from the top right to the bottom left corner is equal to the diagonals of $C,U$ and $C^t.$ We will show that the entries on this diagonal are nonzero and the entries under this line are zero.

Indeed, a straightforward computation shows that the matrix $U=(u_{i,j})$ is a square matrix of size $a-b+1$ with
\begin{align*}
u_{i,j} =(v_i\cdot v_j)\circ F &=(x^{i+j-2}y^{2k+2-i-j})\circ F \\
&=\begin{cases}
y &\text{if}\quad i+j=a-b+2\\
0 &\text{if}\quad i+j\geq a-b+3
\end{cases}
\end{align*}
since $i+j-2\leq 2a-2b\leq 2a-(a+3)=a-3.$ 
Similarly, the matrix $C=(c_{i,j})$ is a square matrix of size $b-\beta$ with
\begin{align*}
c_{i,j} =(u_i\cdot w_j)\circ F &=\begin{cases}
(x^{2k-b+\beta}y^{b-\beta-1}z)\circ F &\text{if}\quad i+j=b-\beta+1\\
(x^{2k+1-i-j}y^{i+j-2}z)\circ F &\text{if}\quad i+j\geq b-\beta+2
\end{cases}\\
&=\begin{cases}
x &\text{if}\quad i+j=b-\beta+1\\
0 &\text{if}\quad i+j\geq b-\beta+2.
\end{cases}
\end{align*}

It is easy to see that $W=\pmt{(w_i\cdot w_j)\circ F}=0$ because $w_i\cdot w_j$ contains $z^2.$ Finally,  $V=(v_{i,j})$ is a matrix of size $(a-b+1)\times (b-\beta)$ with
\begin{align*}
v_{i,j}=(v_i\cdot w_j)\circ F &=(x^{i-1} y^{k+1-i}\cdot x^{k-j} y^{j-1} z)\circ F\\
&=(x^{k+i-j-1}y^{k-i+j}z)\circ F.
\end{align*}
Notice that $ k>b-\beta -1$. Hence if $i\leq j$, then $k-i+j>b-\beta -1.$  Thus $(v_i\cdot w_j)\circ F=0.$ If $i>j,$ then put $\ell:=i-j$, hence $1\leq \ell \leq a-b\leq \beta -2.$ In this case, we will see that $k-i+j> b-\beta -1.$ Indeed, one has
\begin{align*}
k-i+j> b-\beta -1 \Leftrightarrow 2k-2\ell>2b-2\beta -2,
\end{align*}
where the last inequality follows from the fact that
$$2k-2\ell \geq a+b-\beta -2  -(a-b)-(\beta-2) \geq 2b-2\beta.$$
Thus, we see that  $V=0.$

We thus conclude that the Hessian of $F$ is
\begin{align*}
\Hess^{k}_\mathcal{B}(F)=
\bmt{\ast&\cdots &\ast&\ast& \cdots&\ast&\ast&\cdots&x\\
	\vdots&\cdots&\vdots&\vdots&\cdots&\vdots&\vdots&\cdots&\vdots\\
	\ast&\cdots &\ast&\ast&\cdots&\ast & x&\cdots& 0\\ 
	\ast&\cdots &\ast&\ast&\cdots&y&0&\cdots&0\\
	\vdots&\cdots&\vdots&\vdots&\cdots&\vdots&\vdots&\cdots&\vdots\\
	\ast&\cdots &\ast&y&\cdots&0&0&\cdots&0\\
	\ast&\cdots &x&0&\cdots&0&0&\cdots&0\\
	\vdots&\cdots&\vdots&\vdots&\cdots&\vdots&\vdots&\cdots&\vdots\\
	x&\cdots&0&0&\cdots&0&0&\cdots&0}
\end{align*}
 which has nonzero determinant.
\end{proof}

\begin{pro}\label{Proposition3.6}
Assume  $a\geq b\geq 2$. Then the ideal
$$I_\beta=(x^a, y^b-x^{b-1} z, z^2, x^{a-b+1}y^{b-\beta}, y^{b-\beta}z)\subset R$$
has the WLP, whenever $1\leq \beta \leq \min\{a-b+1,b-1\}$.
\end{pro}
\begin{proof}
Set $A_\beta =R/I_\beta$  and $G= R/(x^a, y^b-x^{b-1} z, z^2)$. By Proposition~\ref{Proposition3.1}, one has
$$ A_1 =\frac{G}{(0\colon_{G} y)}\quad \text{and}\quad A_\beta =\frac{A_{\beta -1}}{(0\colon_{A_{\beta-1}} y)},$$
for all $2\leq \beta\leq b-1$. Notice that if $\beta=a-b\leq b-1$, then the socle degree of $A_{\beta}$ is odd. By Proposition~\ref{Proposition2.7} and Lemma~\ref{lemmakey}, it is enough to show that $A_1$ has the WLP. We consider the following cases:

\noindent{\bf \underline{Case 1: $a+b$ is even.}} Then the socle degree of $G$ is $a+b-1$ which is odd. Since $G$ is an Artinian complete intersection algebra of codimension 3, by \cite[ Corollary 2.4]{HMNW2003}, $G$ has the WLP. It follows that $A_1$ has the WLP by  Proposition~\ref{Proposition2.7}(i).

\noindent{\bf \underline{Case 2: $a+b$ is odd.}}  If $a\geq b+3$, then $A_1$ has the WLP by Proposition~\ref{Proposition2.7}(ii) and Lemma~\ref{LemmahilbertfunctionforG}. It follows that the remain case is where $a=b+1$. More precisely, we have to  show that the ideal
$$I=(x^a, y^{a-1}-x^{a-2} z, z^2, x^{2}y^{a-2}, y^{a-2}z)$$
has the WLP. The socle degree of $R/I$ is $2a-3$, hence $k=a-2.$ By Lemma~\ref{lemmakey}, one has
$$H_{R/I} (a-2)=2a-3.$$
Therefore, we can see that a $K$-linear basis of $[R/I]_{a-2}$ is $\mathcal{B} =\mathcal{B}_1\sqcup \mathcal{B}_2,$
where
\begin{align*}
\mathcal{B}_1 &= \{ u_i =x^{a-1-i} y^{i-1} \mid i=1,2,\ldots,a-1 \}\\
\mathcal{B}_2 &= \{ u_{a+i} =x^{a-3-i} y^{i}z \mid i=0,1,\ldots,a-3 \}.
\end{align*}
On the other hand, the Macaulay dual generator of $R/I$ is
$$F=x^{a-1}y^{a-3}z +xy^{2a-4}.$$
To prove that $R/I$ has the WLP, by Proposition~\ref{Proposition2.8_Hessian}, it is enough to show that $\Hess^{a-2}_\mathcal{B}(F)$
has nonzero determinant. Write 
$$\Hess^{a-2}_\mathcal{B}(F) =\bmt{M_{i,j}},$$
where $M_{i,j}=\pmt{(u_i\cdot u_j)\circ F}$ for all $1\leq i,j\leq 2a-3$.

Notice that $M$ is a square matrix of size $2a-3$. First, we have
\begin{align*}
a_{i,2a-2-i}&=(u_i\cdot u_{2a-2-i})\circ F \\
& =\begin{cases}
(x^{a-2}y^{a-3}z)\circ F &\text{if}\quad 1\leq i\leq a-2\\
(y^{2a-4})\circ F &\text{if}\quad i=a-1\\
(x^{a-2}y^{a-3}z)\circ F&\text{if}\quad a\leq i\leq 2a-3\\
\end{cases}\\
&=x.
\end{align*}
Now we assume that $i+j\geq 2a-1$.  If $1\leq i\leq a-1$ then  $j\geq a$, hence
\begin{align*}
a_{i,j}&=(u_i\cdot u_{j})\circ F \\
&=(x^{3a-i-j-4}y^{i+j-a-1}z)\circ F \\
&=0
\end{align*}
since $i+j-a-1\geq a-2.$ By the symmetry of $M_{i,j}$, we only consider the case where $i,j\geq a$. In this case, we have $a_{i,j}=0$ because $u_i\cdot u_j$ contains $z^2.$ 

In summary, the Hessian of $F$ is
\begin{align*}
\Hess^{k}_\mathcal{B}(F)=
\bmt{\ast&\ast &\cdots &\ast&x\\
	\ast&\ast &\cdots &x&0\\
	\vdots&\vdots &\cdots&\vdots&\vdots\\
	\ast&x &\cdots&0&0\\
	x&0&\cdots&0&0}
\end{align*}
which has nonzero determinant.
\end{proof}

We now state our first main result.
\begin{Theo}\label{Theorem3.6}
Consider the ideal $I$ as in \eqref{IdealGorenstein}. If one of the integers $a,b$ and $c$ is equal to 2, then $R/I$ has the WLP.	
\end{Theo}

\begin{proof}
The theorem was proved for the case $b=2$ in Proposition~\ref{Propositionb=2}. By the symmetry of $x$ and $z$, we can always assume that $a\geq c.$ Now if  $c=2$, then it follows from Propositions~\ref{Propositionforbetalarge} and~\ref{Proposition3.6} that $R/I$ has the WLP. 
\end{proof}

\subsection{The ideal $I$ contains a cubic.}
In this subsection, we consider the case where the ideal $I$ contains a cubic.

The first case we consider is $b=3.$  Denote
$$I_\beta=(x^a, y^3-x^\alpha z^{3-\alpha}, z^c, x^{a-\alpha}y^{3-\beta}, y^{3-\beta}z^{c+\alpha -3})\subset R,$$
with $a\geq c$ and $ 1\leq \alpha,\beta\leq 2$ such that $c+\alpha\geq 4$. The socle degree of $R/I_\beta$ is $D=a+c-\beta$  and the free resolution of $R/I_\beta$ is
\begin{align}\label{freeresolution_for_b=3}
{\footnotesize  \xymatrix{ 0\ar[r] & R(-a-c-3+\beta) \ar[r] & {\begin{array}{c} R(-a-c+\beta)\\\oplus  \\R(-a-3+\beta)\\\oplus \\R(-c-3+\beta) \\  \oplus \\R(-a-3+\alpha)\\\oplus \\R(-c-\alpha)  \end{array}}\ar[r] &  {\begin{array}{c}R(-3)\\\oplus \\ R(-a)\\\oplus \\R(-c)\\ \oplus\\  R(-a-3+\alpha+\beta) \\\oplus \\R(-c-\alpha+\beta) \end{array}}\ar[r] & R\ar[r] & R/I_\beta\longrightarrow 0 }.}
\end{align}

\begin{Lem}\label{LemmahilbertfunctionforGwith_b=3}
Set $G= R/(x^a, y^3-x^{\alpha} z^{3-\alpha}, z^c)$ and $k=\lfloor\frac{a+c}{2}\rfloor$. If $a\geq c\geq 2,$ then
\begin{align*}
H_{G} (k)=\begin{cases}
3c-2&\text{if}\quad a=c\\
3c-1&\text{if}\quad a=c+1\\
3c&\text{if}\quad a\geq c+2.
\end{cases}
\end{align*}
Furthermore, if $a\geq c+4$, then 
\begin{align*}
H_{G} (k) =H_{G} (k-1).
\end{align*}
\end{Lem}
\begin{proof}
Since $G$ is resolved by the Koszul complex and $k-a\leq 0$, we get
\begin{align*}
H_G(k)&=\binom{k+2}{2}-\binom{k-1}{2}-\binom{k-a+2}{2}-\binom{k-c+2}{2}+\binom{k-c-1}{2}\\
&=3k-\binom{k-a+2}{2}-\binom{k-c+2}{2}+\binom{k-c-1}{2}.
\end{align*}
If $a+c$ is odd, then $k=\frac{a+c-1}{2}.$ It follows that
\begin{align*}
H_G(k)&=\begin{cases}
3c-1&\text{if}\quad a=c+1\\
3c&\text{if}\quad a\geq c+3.
\end{cases}
\end{align*}
If $a+c$ is even, then $k=\frac{a+c}{2}.$ Therefore
\begin{align*}
H_G(k)&=\begin{cases}
3c-2&\text{if}\quad a=c\\
3c&\text{if}\quad a\geq c+ 2.
\end{cases}
\end{align*}
Analogously we can check that
\begin{align*}
H_G(k-1)&=3(k-1)-\binom{k-c+1}{2}+\binom{k-c-2}{2}\\
&=\begin{cases}
3c-3&\text{if}\quad a=c \; \text{or}\; a=c+1\\
3c-1&\text{if}\quad  a=c +2\; \text{or}\; a=c+3\\
3c&\text{if}\quad  a\geq c+4.
\end{cases}
\end{align*}
Thus, if $a\geq c+4$, then $H_{G} (k) =H_{G} (k-1).$
\end{proof}

\begin{Lem}\label{LemmahilbertfunctionforGwith_b=3and_beta=1}
Assume $\beta=1$ and $k=\lfloor\frac{a+c-1}{2}\rfloor$. If $a\geq c,$ then
\begin{align*}
H_{R/I_1} (k)=\begin{cases}
3c-3&\text{if}\quad a=c\\
3c-3+\alpha&\text{if}\quad a\geq c+1.
\end{cases}
\end{align*}
Furthermore, if $a\geq c+3$, then 
\begin{align*}
H_{R/I_1} (k) =H_{R/I_1} (k-1).
\end{align*}
\end{Lem}
\begin{proof}
As $k-a<0$ and $c\geq \gamma+1\geq 2,$ it follows from \eqref{freeresolution_for_b=3} with $\beta=1$ that
\begin{align*}
H_{R/I_1}(k)=3k-\binom{k-c+1}{1}-\binom{k-c}{1}-\binom{k-c-\alpha+2}{1}.
\end{align*}
By considering the cases where $a=c+j$ for each  $j\in \{0,1,\ldots,4\} $ or $a\geq c+5$, it is easy to see that
\begin{align*}
H_{R/I_1}(k)=\begin{cases}
3c-3&\text{if}\quad a=c \\
3c-3+\alpha &\text{if}\quad  a\geq c+1.
\end{cases}
\end{align*}
A straightforward computations also shows that
\begin{align*}
H_{R/I_1}(k-1)=\begin{cases}
3c-6&\text{if}\quad a=c \\
3c-3&\text{if}\quad  a=c +1\; \text{or}\; a=c+2\\
3c-3+\alpha &\text{if}\quad  a\geq c+3.
\end{cases}
\end{align*}
Thus if $a\geq c+3$ then $H_{R/I_1}(k)=H_{R/I_1}(k-1).$
\end{proof}
The following is useful to prove the next results. Recall that the determinant of block matrices can be computed as follows: suppose $A, B,C$ and $D$ are  matrices  of size $n\times n, n\times m,m\times n$, and $m\times m$, respectively. Then
\begin{align}\label{determinantofBlockmatrix}
\det \bmt{ A &\vdots & B\\\cdots & \cdots &\cdots\\ C & \vdots & D} =\begin{cases}
\det (A) \det (D-CA^{-1}B)  &\text{if}\; A \;\text{is invertible}\\
\det (D) \det (A-BD^{-1}C) &\text{if}\; D \;\text{is invertible}.
\end{cases}
\end{align}
\begin{pro}\label{Propositionforb=3}
Consider the ideal
$$I_\beta=(x^a, y^3-x^\alpha z^{3-\alpha}, z^c, x^{a-\alpha}y^{3-\beta}, y^{3-\beta}z^{c+\alpha -3}),$$
with $a\geq c$ and $ 1\leq \alpha,\beta\leq 2$ such that $c+\alpha\geq 4$.  Then $R/I_\beta$ has the WLP.
\end{pro}

\begin{proof}
Recall that $G=R/(x^a, y^3-x^\alpha z^{3-\alpha}, z^c)$ is a complete intersection of codimension 3, hence it has the WLP.  We consider the following two cases:

\noindent{\bf \underline{Case 1: $a+c$ is odd}.}
In this case, the socle degree of $G$ is odd. Hence, $R/I_1$ has the WLP by Propositions~\ref{Proposition2.7}(i) and~\ref{Proposition3.1}. Now if $a\geq c+3$, then $R/I_2$ also has the WLP by Lemma~\ref{LemmahilbertfunctionforGwith_b=3and_beta=1} and Proposition~\ref{Proposition2.7}(ii).
Thus, we only need to prove that 
$$I_2=(x^a, y^3-x^\alpha z^{3-\alpha}, z^{a-1}, x^{a-\alpha}y, yz^{a+\alpha-4})$$
 has the WLP. In this case, one has  $D=2a-3$ and $k=a-2$. 

\noindent{\bf \underline{Subcase 1: $\alpha=1$}.} Set $L=x-y+z$. By Proposition~\ref{Proposition2.5}, it suffices to show that
$$\times L: [R/I_2]_{a-2}\longrightarrow [R/I_2]_{a-1}$$
is an isomorphism, or equivalently $[R/(I_2,L)]_{a-1}=0.$ We have
\begin{align*}
R/(I_2,L)\cong K[x,z]/J,
\end{align*}
where  $J=(x^a,x^3+3x^2z+2xz^2+z^3,z^{a-1},x^{a-1}z, xz^{a-3}+z^{a-2}).$
We will prove that $[K[x,z]/J]_{a-1}=0$, or equivalent $x^iz^{a-1-i}\in J$ for all $0\leq i\leq a-1.$ We do it by induction on $i$. We first see that $xz^{a-2}$ and $x^2z^{a-3}\in J$ since $z^{a-1},xz^{a-3}+z^{a-2}\in J$. For any $i\geq 3$, one has
\begin{align*}
x^iz^{a-1-i}&=x^3x^{i-3}z^{a-1-i} =-(z^3+2xz^2+3x^2z)x^{i-3}z^{a-1-i} \\
&=-x^{i-3}z^{a+2-i}-3x^{i-2}z^{a+1-i}-3x^{i-1}z^{a-i}\in J,
\end{align*}
by the induction hypothesis .

\noindent{\bf \underline{Subcase 2: $\alpha=2$}.} In this case,
$$I_2=(x^a, y^3-x^2 z, z^{a-1}, x^{a-2}y, yz^{a-2}).$$
It follows from \eqref{freeresolution_for_b=3} that
\begin{align*}
H_{R/I_2}(a-2) =H_{R/I_2}(a-1)=\binom{a}{2}-\binom{a-3}{2}=3a-6.
\end{align*}
A $K$-linear basis of $[R/I_2]_{a-2}$ is
$\mathcal{B} =\mathcal{B}_1\sqcup \mathcal{B}_2\sqcup \mathcal{B}_3,$
where
\begin{align*}
\mathcal{B}_1 &= \{ u_i =x^{a-1-i} z^{i-1} \mid i=1,2,\ldots,a-1\}\\
\mathcal{B}_2 &= \{v_{i} =x^{a-2-i} yz^{i-1} \mid i=1,2,\ldots,a-2    \}\\
\mathcal{B}_3 &= \{ w_{i} =x^{a-3-i} y^2 z^{i-1} \mid i=1,2,\ldots, a-3 \}
\end{align*}
and a $K$-linear basis of $[R/I_2]_{a-1}$ is
$\mathcal{B}^\prime =\mathcal{B}_1^\prime\sqcup \mathcal{B}_2^\prime\sqcup \mathcal{B}_3^\prime,$
where
\begin{align*}
\mathcal{B}_1^\prime &= \{ u_i^\prime =x^{a-i} z^{i-1} \mid i=1,2,\ldots,a-1\}\\
\mathcal{B}_2^\prime &= \{v_{i}^\prime =x^{a-2-i} y^2z^{i-1} \mid i=1,2,\ldots,a-2    \}\\
\mathcal{B}_3^\prime &= \{ w_{i}^\prime =x^{a-2-i} y z^{i} \mid i=1,2,\ldots, a-3 \}.
\end{align*}
Set $L=x+y+z$. By Proposition~\ref{Proposition2.5}, it is enough to show that
$$\times L: [R/I_2]_{a-2}\longrightarrow [R/I_2]_{a-1}$$
is an isomorphism, or equivalently the matrix representation $M$ of $\times L$ with respect to these bases  has nonzero determinant. 
It is easy to see that
\begin{align*}
\times L(u_i)&=x^{a-i}z^{i-1} +x^{a-1-i}z^{i} +x^{a-1-i}yz^{i-1}\\
\times L(v_i)&=x^{a-1-i}yz^{i-1} +x^{a-2-i}yz^{i} +x^{a-2-i}y^2z^{i-1}\\
\times L(w_i)&=x^{a-1-i}z^{i} +x^{a-2-i}y^2z^{i-1} +x^{a-3-i}y^2z^{i}
\end{align*}
since $y^3=x^2z$ in $R/I_2$. It follows that
\begin{align*}
M=\bmt{ A &\vdots &0&\vdots & B^t\\\cdots & \cdots &\cdots&\cdots&\cdots\\ 0 & \vdots & \Id & \vdots & C^t\\\cdots & \cdots &\cdots&\cdots&\cdots\\ B & \vdots &C& \vdots & 0},
\end{align*}
where $A,B$ and $C$ are matrices of size $(a-1)\times(a-1),(a-3)\times(a-1)$ and  $(a-3)\times(a-2)$, respectively
\begin{align*}
A=\bmt{1&0&\cdots & 0 &0 \\1&1&\cdots & 0 &0 \\\vdots &\vdots &\ddots & \vdots &\vdots\\ 0&0&\cdots & 1 &0\\ 0&0&\cdots & 1 &1 },\; B=\left[
\begin{array}{c| c c c c c | c}
0& 1&0&\cdots &0&0&0\\
0& 0&1&\cdots &0&0&0\\
\vdots &\vdots &\vdots &\ddots & \vdots &\vdots &\vdots \\
0& 0&0&\cdots &1&0&0\\
0& 0&0&\cdots &0&1&0\\
\end{array}
\right], \; C=\left[
\begin{array}{c c c c  c | c}
1& 1&\cdots &0&0&0\\
0& 1&\cdots &0&0&0\\
\vdots &\vdots  &\ddots & \vdots &\vdots &\vdots \\
0& 0&\cdots &1&1&0\\
0& 0&\cdots &0&1&1\\
\end{array}
\right].
\end{align*}
Set
$N=\bmt{ A &\vdots &0&\\\cdots & \cdots &\cdots\\ 0 & \vdots & \Id}.$
Then $\det(N)=1$ and a computation shows that
\begin{align*}
P& =\bmt{B&\vdots &C}\bmt{ A^{-1} &\vdots &0&\\\cdots & \cdots &\cdots\\ 0 & \vdots & \Id}\bmt{B^t\\\cdots \\C^t} =B A^{-1}B^t+ C C^t\\
& =\left[\begin{array}{c c c c c c c}
3&1&0&\cdots &0 &0 &0\\ 0&3&1&\cdots &0 &0 &0 \\ 1&0&3&\cdots &0 &0 &0 \\ \vdots&\vdots&\vdots&\ddots &\vdots &\vdots &\vdots\\
(-1)^{a-6}&(-1)^{a-7}&(-1)^{a-8}&\cdots &3 &1 &0 \\(-1)^{a-5}&(-1)^{a-6}&(-1)^{a-7}&\cdots &0 &3 &1 \\ (-1)^{a-4}&(-1)^{a-5}&(-1)^{a-6}&\cdots &1&0 &3
\end{array}\right]
\end{align*}
has nonzero determinant. Thus , by \eqref{determinantofBlockmatrix}, $\det(M)=\det(N)\det (-P)\neq 0$.

\noindent{\bf \underline{Case 2: $a+c$ is even}.} In this case, the socle degree of $G$ is even. By Propositions~\ref{Proposition2.7}(i) and~\ref{Proposition3.1}, it is enough to show that $R/I_1$ has the WLP. If $a\geq c+4$, then $R/I_1$ has the WLP by Proposition~\ref{Proposition2.7}(ii) and Lemma~\ref{LemmahilbertfunctionforGwith_b=3}.
It remains to consider the cases where $a=c$ or $a=c+2.$

\noindent{\bf \underline{Subcase 1: $a=c+2$}.} In this case, $k=a-2.$ Set $L=x-y+z$. By Proposition~\ref{Proposition2.5}, it suffices to show that
$$\times L: [R/I_1]_{a-2}\longrightarrow [R/I_1]_{a-1}$$
is an isomorphism, or equivalently $[R/(I_1,L)]_{a-1}=0.$ We have
\begin{align*}
R/(I_1,L)\cong K[x,z]/J,
\end{align*}
where  
$$J=(x^a,x^3+3x^2z+3xz^2+z^3-x^{\alpha}z^{3-\alpha},z^{a-2},x^{a-\alpha}z^2+2x^{a+1-\alpha}z, x^2z^{a+\alpha-5}+2xz^{a+\alpha-4}).$$
We will prove that $[K[x,z]/J]_{a-1}=0$, or equivalently $x^iz^{a-1-i}\in J$ for all $0\leq i\leq a-1.$ We do it by induction on $i$. It is easy to see that $z^{a-1}$ and $xz^{a-2}\in J$ since $z^{a-2}\in J$ and $x^2z^{a-3}\in J$ since $x^2z^{a+\alpha-5}+2xz^{a+\alpha-4}\in J$. For any $i\geq 3$, one has
\begin{align*}
x^iz^{a-1-i}&=x^3x^{i-3}z^{a-1-i} =-(z^3+3xz^2+3x^2z-x^\alpha z^{3-\alpha})x^{i-3}z^{a-1-i} \\
&=-x^{i-3}z^{a+2-i}-3x^{i-2}z^{a+1-i}-3x^{i-1}z^{a-i}+x^{i+\alpha-3}z^{a+2-\alpha-i}\in J,
\end{align*}
by the induction hypothesis .

\noindent{\bf \underline{Subcase 2: $a=c$}.} By symmetry of $x$ and $z$, we can assume $\alpha=1$. More precisely, consider the ideal
$$I_1=(x^a, y^3-x z^{2}, z^{a}, x^{a-1}y^2, y^2z^{a-2}).$$
It follows that $k=a-1$ and 
$$H_{R/I_1}(a)= H_{R/I_1}(a-1) =3a-3$$
by Lemma~\ref{LemmahilbertfunctionforGwith_b=3and_beta=1}. It is easy to see that  $[R/I_1]_{a-1}$ has a basis $\mathcal{B} =\mathcal{B}_1\sqcup \mathcal{B}_2\sqcup \mathcal{B}_3,$
 where
 \begin{align*}
 \mathcal{B}_1 &= \{ u_i =x^{a-i} z^{i-1} \mid i=1,2,\ldots,a\}\\
 \mathcal{B}_2 &= \{ v_{i} =x^{a-1-i} y z^{i-1} \mid i=1,2,\ldots, a-1  \}\\
 \mathcal{B}_3 &= \{ w_{i} =x^{a-2-i} y^2z^{i-1} \mid i=1,2,\ldots,a-2 \}
 \end{align*}
 and $[R/I_1]_{a}$ has a basis $\mathcal{B}^\prime =\mathcal{B}_1^\prime\sqcup \mathcal{B}_2^\prime\sqcup \mathcal{B}_3^\prime,$
 where
 \begin{align*}
 \mathcal{B}_1^\prime &= \{ u_i^\prime =x^{a-i}y z^{i-1} \mid i=1,2,\ldots,a\}\\
 \mathcal{B}_2^\prime &= \{ v_{i}^\prime =x^{a-i} z^{i} \mid i=1,2,\ldots, a-1  \}\\
 \mathcal{B}_3^\prime &= \{ w_{i}^\prime =x^{a-1-i} y^2z^{i-1} \mid i=1,2,\ldots,a-2 \}.
 \end{align*}
  Set $L=x+y+z$. By Proposition~\ref{Proposition2.5}, it is enough to show that
 $$\times L: [R/I_1]_{a-1}\longrightarrow [R/I_1]_{a}$$
 is an isomorphism, or equivalently the matrix representation $M$ of $\times L$ with respect to these bases  has nonzero determinant. 
 It is easy to see that
 \begin{align*}
 \times L(u_i)&=x^{a+1-i}z^{i-1} +x^{a-i}z^{i} +x^{a-i}yz^{i-1}\\
 \times L(v_i)&=x^{a-i}yz^{i-1} +x^{a-1-i}yz^{i} +x^{a-1-i}y^2z^{i-1}\\
 \times L(w_i)&=x^{a-1-i}y^2z^{i-1} +x^{a-2-i}y^2z^{i} +x^{a-1-i}z^{i+1},
 \end{align*}
 as $y^3=x z^2$ in $R/I_1$. It follows that 
 \begin{align*}
 M=\bmt{ \Id &\vdots &A^t&\vdots & 0\\\cdots & \cdots &\cdots&\cdots&\cdots\\ A & \vdots & 0 & \vdots & C\\\cdots & \cdots &\cdots&\cdots&\cdots\\ 0& \vdots &B& \vdots & D},
 \end{align*}
 where $A,B,C$ and $D$ are the matrices of size $(a-1)\times a,(a-2)\times(a-1), (a-1)\times(a-2)$ and $(a-2)\times(a-2)$ respectively, where 
 \begin{align*}
 &A=\left[
 \begin{array}{c c c c  c | c}
 1& 1&\cdots &0&0&0\\
 0& 1&\cdots &0&0&0\\
 \vdots &\vdots  &\ddots & \vdots &\vdots &\vdots \\
 0& 0&\cdots &1&1&0\\
 0& 0&\cdots &0&1&1\\
 \end{array}
 \right], \quad B=\left[
 \begin{array}{ c c c c c | c}
  1&0&\cdots &0&0&0\\
  0&1&\cdots &0&0&0\\
  \vdots &\vdots &\ddots & \vdots &\vdots &\vdots \\
  0&0&\cdots &1&0&0\\
  0&0&\cdots &0&1&0\\
 \end{array}
 \right],\\
 &C=\left[
 \begin{array}{ c c c c c} 
 0&0&\cdots&0&0\\
 \hline 1&0&\cdots &0&0\\
 0&1&\cdots &0&0\\
 \vdots &\vdots &\ddots & \vdots &\vdots \\
 0&0&\cdots &1&0\\
 0&0&\cdots &0&1\\
 \end{array}
 \right],  \;\qquad D=\bmt{1&0&\cdots & 0 &0 \\1&1&\cdots & 0 &0 \\\vdots &\vdots &\ddots & \vdots &\vdots\\ 0&0&\cdots & 1 &0\\ 0&0&\cdots & 1 &1 }.
 \end{align*}
 It follows from \eqref{determinantofBlockmatrix} that
 \begin{align*}
 \det(M)&=\det \bmt{ -AA^t&\vdots & C\\ \cdots&\cdots&\cdots\\   B& \vdots & D}\\
 &=\det  \big(-AA^t-CD^{-1}B \big).
  \end{align*} 
 A computation as in the proof of the above subcase 2 in Case 1 shows that the matrix $P=AA^t+CD^{-1}B$ has nonzero determinant, hence $\det(M)\neq 0.$ 
 \end{proof}

Now we consider the case where $c=3$. More precisely, consider
$$I_\beta=(x^a, y^b-x^\alpha z^\gamma, z^3, x^{a-\alpha}y^{b-\beta}, y^{b-\beta}z^{3-\gamma})\subset R,$$
where $1\leq \alpha\leq a-1,\, 1\leq \beta\leq b-1$ and $1\leq \gamma\leq 2$ such that $\alpha+\gamma =b.$ It is clear that $a\geq b-1$ and $a\geq 3$.  

First, we have the following.
\begin{Lem}\label{LemmahilbertfunctionforGwith_3=3}
Set $G= R/(x^a, y^b-x^{\alpha} z^{b-\alpha}, z^3)$ and $k=\lfloor\frac{a+b}{2}\rfloor$. If $a\geq b-1,$ then
\begin{align*}
H_{G} (k)=\begin{cases}
3b-4&\text{if}\quad a=b-1\\
3b-2&\text{if}\quad a=b\\
3b-1&\text{if}\quad a=b+1\\
3b&\text{if}\quad a\geq b+2.
\end{cases}
\end{align*}
Furthermore, if $a\geq b+4$, then 
\begin{align*}
H_{G} (k) =H_{G} (k-1).
\end{align*}
\end{Lem}
\begin{proof}
The proof proceeds along the same lines as in Lemma~\ref{LemmahilbertfunctionforGwith_b=3}.
\end{proof}

Recall that the socle degree of $R/I_\beta$ is $D=a+b-\beta$ and $k=\lfloor \frac{a+b-\beta}{2}\rfloor.$   Since the free resolution of $R/I_\beta$ is
\begin{align}\label{resolutionforb=3}
{\footnotesize  \xymatrix{ 0\ar[r] & R(-a-b-3+\beta) \ar[r] & {\begin{array}{c} R(-a-b+\beta)\\ \oplus \\R(-a-3+\beta)\\\oplus \\R(-b-3+\beta)\\\oplus \\R(-a-\gamma)\\ \oplus \\R(-3-\alpha) \end{array}}\ar[r]&  {\begin{array}{c} R(-a)\\ \oplus \\R(-b)\\\oplus \\R(-3)\\\oplus \\R(-a-\gamma+\beta)\\ \oplus \\R(-3-\alpha+\beta ) \end{array}}\ar[r] & R \ar[r]&R/I_\beta\ar[r]& 0,}}
\end{align}
we can determine the Hilbert function of $R/I_\beta$ in degree $k$ as follows.
\begin{Lem}\label{lemmakeyforc=3}
For every $1\leq \beta\leq b-1$. Set $k=\lfloor \frac{a+b-\beta}{2}\rfloor.$ 
\begin{enumerate}
\item[\rm (1)] If $\gamma =1$, then $a\geq b$ and 
\begin{align*}
H_{R/I_\beta} (k)&=\begin{cases} 
3b-3&\text{if}\quad a=b \; \text{and}\; \beta =1\\
3b-\beta&\text{if}\quad \beta \leq a-b\\
3b-\beta-1&\text{if}\quad \beta =a-b+1\geq 2\\
2a+b-3\beta+2&\text{if}\quad \beta\geq a-b+2.
\end{cases}
\end{align*}

\item [\rm (2)] If $\gamma =2$, then
\begin{align*}
H_{R/I_\beta} (k)&=\begin{cases} 
3b-4&\text{if}\quad a=b-1 \; \text{and}\; \beta =1\\
3b-3&\text{if}\quad a=b \; \text{and}\; \beta =1\\
3b-2\beta&\text{if}\quad \beta \leq a-b+1\; \text{and}\; a\neq b\\
a+2b-3\beta+2&\text{if}\quad \beta\geq a-b+2 \geq 2.
\end{cases}
\end{align*}
\item [\rm (3)] Furthermore, if $1\leq \beta\leq a-b-2$, then 
\begin{align*}
H_{R/I_\beta} (k) =H_{R/I_\beta} (k-1).
\end{align*}
\end{enumerate}
\end{Lem}
\begin{proof}
Notice that $k<a+b-\beta$ and $k<a+\gamma$.  It follows from \eqref{resolutionforb=3} that 
{\small \begin{align*} 
H_{R/I_\beta} (k) = &\binom{k+2}{2} -\binom{k-1}{2}  -\binom{k-a+2}{2}-\binom{k-b+2}{2}-\binom{k-a-\gamma+\beta+2}{2}\\
&-\binom{k-\alpha+\beta-1}{2}+\binom{k-a+\beta-1}{2}+\binom{k-b+\beta-1}{2}+\binom{k-\alpha-1}{2}.
\end{align*}}
If $\gamma=1$ then $a\geq b$ and 
{\footnotesize   \begin{align*} 
H_{R/I_\beta}(k)=3k -\binom{k-b+1}{1}-\binom{k-b}{1}-\binom{k-b+\beta-1}{1}-\binom{k-a+\beta}{1}-\binom{k-a+\beta-1}{1}.
\end{align*}}
If $\gamma=2$, then
\begin{align*} 
H_{R/I_\beta}(k)&
=\begin{cases}
3b-4&\text{if} \ a=b-1\ \text{and}\ \beta=1\\
3k -\binom{k-b+1}{1}-\binom{k-b+\beta}{1}-\binom{k-b+\beta-1}{1}-\binom{k-a+\beta-1}{1}&\text{otherwise}.
\end{cases}
\end{align*}

Firstly, if $a+b-\beta$ is even, then $k=\frac{a+b-\beta}{2}$.  

\noindent {\bf \underline{Case 1: $\beta\geq a-b+2$.}} If  $\beta\geq a-b+6$, then $k-a+\beta \geq 3,\ k-b+\beta\geq 2$ and $k-b+3\leq 0$, hence
\begin{align*}
H_{R/I_\beta}(k)=\begin{cases}
2a+b-3\beta+2&\text{if}\quad \gamma=1\\
a+2b-3\beta+2&\text{if}\quad \gamma=2.
\end{cases}
\end{align*}
 A direct computation for the cases where $\beta=a-b+2$ or $\beta=a-b+4$ also shows that
 \begin{align*}
 H_{R/I_\beta}(k)=\begin{cases}
 2a+b-3\beta+2&\text{if}\quad \gamma=1\\
 a+2b-3\beta+2&\text{if}\quad \gamma=2.
 \end{cases}
 \end{align*}

\noindent {\bf \underline{Case 2: $\beta\leq a-b$.}} If $\beta\leq a-b-2$, then  $k-a+\beta +1\leq 0$ and $k-b\geq 1$, hence a straightforward computation shows that
\begin{align*}
H_{R/I_\beta}(k)=\begin{cases}
3b-\beta&\text{if}\quad \gamma=1\\
3b-2\beta&\text{if}\quad \gamma=2.
\end{cases}
\end{align*}
If $\beta=a-b$, then a simple computation shows that 
\begin{align*}
H_{R/I_\beta}(k)=\begin{cases}
3b-\beta&\text{if}\quad \gamma=1\\
3b-2\beta&\text{if}\quad \gamma=2.
\end{cases}
\end{align*}

Secondly, if $a+b-\beta$ is odd, then $k=\frac{a+b-\beta-1}{2}$.  The proof is similar to the case  $a+b-\beta$ even.

Finally, a similar computation also shows that  if $1\leq \beta\leq a-b-2$, then 
\begin{align*}
H_{R/I_\beta} (k) =H_{R/I_\beta} (k-1).
\end{align*}
\end{proof}

\begin{pro} \label{Propositionforbetasmallandc=3}
Assume that  $a\geq b+1,\ 1\leq \alpha\leq a-1$  and $1\leq \gamma\leq 2$ such that $\alpha+\gamma=b$. If   $1\leq \beta\leq  \min\{a-b,b-1\}$, then the ideal
$$I_\beta=(x^a, y^b-x^{\alpha} z^\gamma, z^3, x^{a-\alpha}y^{b-\beta}, y^{b-\beta}z^{3-\gamma})\subset R$$
has the WLP.
\end{pro}
\begin{proof}
Set $G= R/(x^a, y^b-x^{\alpha} z^{\gamma}, z^3)$ and $A_\beta=R/I_\beta.$ By Proposition~\ref{Proposition3.1}, one has
$$ A_1 =\frac{G}{(0\colon_{G} y)}\quad \text{and}\quad A_\beta =\frac{A_{\beta -1}}{(0\colon_{A_{\beta-1}} y)},$$
for all $2\leq \beta\leq b-1$. Notice that if $a\geq b+2$ and $\beta=a-b-1\leq b-1$, then the socle degree of $A_{\beta}$ is odd. By Proposition~\ref{Proposition2.7} and Lemma~\ref{lemmakeyforc=3}, it is enough to show that $A_1$ has the WLP. We consider the following cases:

\noindent{\bf \underline{Case 1: $a+b$ is odd.}} Then the socle degree of $G$ is odd. Since $G$ is an Artinian complete intersection algebra of codimension 3, $G$ has the WLP \cite[ Corollary 2.4]{HMNW2003}. Thus it follows that $A_1$ has the WLP by  Proposition~\ref{Proposition2.7}(i).

\noindent{\bf \underline{Case 2: $a+b$ is even.}}  If $a-b\geq 4$, then $A_1$ has the WLP by Proposition~\ref{Proposition2.7}(ii) and Lemma~\ref{LemmahilbertfunctionforGwith_3=3}. As $a+b$ is even, hence we only consider the case where $a=b+2.$ Firstly, we will prove that $A_1$ has the WLP for $\gamma=1.$  More precisely, consider the ideal
$$I=(x^{b+2}, y^b-x^{b-1} z, z^3, x^{3}y^{b-1}, y^{b-1}z^2)\subset R.$$
In this case, we have  $k=b$ and $H_{R/I}(b)= 3b-1$ by Lemma~\ref{lemmakeyforc=3}.
A $K$-linear basis of $[R/I]_{b}$ is $\mathcal{B} =\mathcal{B}_1\sqcup \mathcal{B}_2\sqcup \BB_3,$
where
\begin{align*}
\mathcal{B}_1 &= \{ u_i =x^{b+1-i} y^{i-1} \mid i=1,2,\ldots,b+1 \}\\
\mathcal{B}_2 &= \{ u_{b+1+i} =x^{b-1-i} y^{i}z \mid i=1,2,\ldots,b-1 \}\\
\mathcal{B}_3 &= \{ u_{2b+i} =x^{b-1-i} y^{i-1}z^2 \mid i=1,2,\ldots,b-1 \}.
\end{align*}
On the other hand, the Macaulay dual generator of $R/I$ is
$$F=x^{b+1}y^{b-2}z^{2}+ x^{2}y^{2b-2}z+x^{3-b}y^{3b-2},$$
where the last monomial does not appear in $F$  if $b>3.$ To prove the proposition, by Proposition~\ref{Proposition2.8_Hessian}, it is enough to show that $\Hess^{b}_\mathcal{B}(F)=\bmt{M_{i,j}}$
has nonzero determinant. 
A straightforward computation shows that
\begin{align*}
M_{i,j}&=(u_i\cdot u_j)\circ F=\begin{cases}
x&\text{if}\; i+j=3b\\
0&\text{if}\; i+j\geq 3b+1.
\end{cases}
\end{align*}
It follows that $\Hess^{b}_\mathcal{B}(F)$
has nonzero determinant. 

It remains to show that $R/I$ has the WLP for the case $\gamma=2$. In this case, we have
$$I=(x^{b+2}, y^b-x^{b-2} z^2, z^3, x^{4}y^{b-1}, y^{b-1}z).$$
It follows that $k=b$ and 
$$H_{R/I}(b)=H_{R/I}(b+1)=3b-2$$
by Lemma~\ref{lemmakeyforc=3}.  Set $L=x+y+z$. By Proposition~\ref{Proposition2.5}, it is enough to show that
$$\times L: [R/I]_{b}\longrightarrow [R/I]_{b+1}$$
is an isomorphism. To do it, let $\mathcal{B}$ and $\mathcal{B}^\prime$ be the $K$-linear bases of $[R/I]_{b}$ and  $[R/I]_{b+1}$, respectively and let $M$ be the matrix representation of $\times L$ with respect to these bases. Write $\mathcal{B} =\mathcal{B}_1\sqcup \mathcal{B}_2\sqcup \BB_3,$
where
\begin{align*}
\mathcal{B}_1 &= \{ u_i =x^{b+1-i} y^{i-1} \mid i=1,2,\ldots,b \}\\
\mathcal{B}_2 &= \{ v_i=x^{b-i} y^{i-1}z \mid i=1,2,\ldots,b-1 \}\\
\mathcal{B}_3 &= \{ w_i =x^{b-1-i} y^{i-1}z^2 \mid i=1,2,\ldots,b-1 \}
\end{align*}
and
$\mathcal{B}^\prime =\mathcal{B}_1^\prime\sqcup \mathcal{B}_2^\prime\sqcup \BB_3^\prime,$
where
\begin{align*}
\mathcal{B}_1^\prime &= \{ u_i^\prime =x^{b+2-i} y^{i-1} \mid i=1,2,\ldots,b \}\\
\mathcal{B}_2^\prime &= \{ v_i^\prime=x^{b+1-i} y^{i-1}z \mid i=1,2,\ldots,b-1 \}\\
\mathcal{B}_3^\prime &= \{ w_i^\prime =x^{b-i} y^{i-1}z^2 \mid i=1,2,\ldots,b-1 \}.
\end{align*}
It is easy to see that
\begin{align*}
\times L(u_i)&=x^{b+2-i}y^{i-1} +x^{b+1-i}y^{i} +x^{b+1-i}y^{i-1}z\\
\times L(v_i)&=x^{b+1-i}y^{i-1}z +x^{b-i}y^{i}z +x^{b-i}y^{i-1}z^2\\
\times L(w_i)&=x^{b-i}y^{i-1}z^2 +x^{b-1-i}y^{i}z^2
\end{align*}
and thus $M$ is a lower trianguler matrix and all the entries on the main diagonal are one. It follows that $\det(M)=1.$ 
\end{proof}

\begin{pro}\label{Proposition3.13}
Assume that  $ 1\leq \alpha\leq a-1, 1\leq \beta\leq b-1$  and $1\leq \gamma\leq 2$ such that $\alpha+\gamma=b$. If   $a\leq 2b-2$, then the ideal
$$I_\beta=(x^a, y^b-x^{\alpha} z^\gamma, z^3, x^{a-\alpha}y^{b-\beta}, y^{b-\beta}z^{3-\gamma})\subset R$$
has the WLP, whenever $a-b+1\leq \beta\leq  b-1$.
\end{pro}
\begin{proof}	
Notice that the socle degree of $R/I_\beta$ is $D=a+b-\beta.$ Therefore, if $\beta=a-b+1\geq 1$, then $ D=2b-1$ is odd. To prove that $R/I$ has the WLP for all $a-b+1\leq\beta \leq b-1$, by Propositions~\ref{Proposition2.7}(i) and~\ref{Proposition3.1}, it is enough to prove that $R/I$ has the WLP whenever $a+b-\beta$ is odd.

Now let $\beta\geq 1$ be an integer such that $a-b+1\leq\beta \leq b-1$ and $a+b-\beta$ is odd. In this case, one has $k=\frac{a+b-\beta-1}{2}$. We will prove that $R/I$ has the WLP by considering the following cases:

\noindent{\bf \underline{Case 1: $\beta= a-b+1$.}} There are the following two subcases.

\noindent{\bf \underline{Subcase 1: $\beta=1$.}} More precisely, we consider the ideal
\begin{align*}
I=\begin{cases}
(x^a,y^a-x^{a-1} z,z^3,xy^{a-1},y^{a-1}z^{2})&\text{if}\quad \gamma=1\\
(x^a,y^a-x^{a-2} z^{2},z^3,x^2y^{a-1},y^{a-1}z)&\text{if}\quad \gamma=2.
\end{cases}
\end{align*}
Hence $k=a-1$ and by Lemma~\ref{lemmakeyforc=3}, one has
$$H_{R/I}(a-1)= H_{R/I}(a) =3a-3.$$ 
Set $L= x+y+z$.  By Proposition~\ref{Proposition2.5}, it is enough to show that
$$\times L: [R/I]_{a-1}\longrightarrow [R/I]_{a}$$
is an isomorphism. To do it, let $\mathcal{B}$ and $\mathcal{B}^\prime$ be the $K$-linear bases of $[R/I]_{a-1}$ and $[R/I]_{a}$, respectively and we will prove that the matrix representation $M$ of $\times L$ with respect to these bases  has nonzero determinant.  We do it for the case $\gamma =1$. The case $\gamma=2$ is similarly proved. In this case, we write $\mathcal{B} =\mathcal{B}_1\sqcup \mathcal{B}_2\sqcup \BB_3,$
where 
\begin{align*}
\mathcal{B}_1 &= \{ u_{i} =x^{a-i} y^{i-1} \mid i=1,2,\ldots,a \}\\
\mathcal{B}_2&= \{ v_i =x^{a-1-i} y^{i-1}z \mid i=1,2,\ldots,a-1 \}\\
\mathcal{B}_3&= \{ w_i =x^{a-2-i} y^{i-1}z^2 \mid i=1,2,\ldots,a-2 \}.
\end{align*}
and  $\mathcal{B}^\prime =\mathcal{B}_1^\prime\sqcup \mathcal{B}_2^\prime\sqcup \BB_3^\prime,$
where 
\begin{align*}
\mathcal{B}_1^\prime  &= \{ u_{i}^\prime  =x^{a-i} y^{i-1}z \mid i=1,2,\ldots,a \}\\
\mathcal{B}_2^\prime &= \{ v_i^\prime  =x^{a-1-i} y^{i-1}z^2 \mid i=1,2,\ldots,a-1 \}\\
\mathcal{B}_3^\prime &= \{ w_i^\prime  =x^{a-i} y^{i} \mid i=1,2,\ldots,a-2 \}.
\end{align*}
It is easy to see that
\begin{align*}
\times L(u_i)&=x^{a+1-i}y^{i-1} +x^{a-i}y^{i} +x^{a-i}y^{i-1}z\\
\times L(v_i)&=x^{a-i}y^{i-1}z +x^{a-1-i}y^{i}z +x^{a-1-i}y^{i-1}z^2\\
\times L(w_i)&=x^{a-1-i}y^{i-1}z^2 +x^{a-2-i}y^{i}z^2.
\end{align*}
It follows that
\begin{align*}
M=\bmt{ \Id &\vdots & A_{12}&\vdots & 0\\\cdots & \cdots &\cdots&\cdots&\cdots\\ 0 & \vdots & \Id& \vdots & A_{23}\\\cdots & \cdots &\cdots&\cdots&\cdots\\ A_{31} & \vdots &0& \vdots & 0},
\end{align*}
where 
{\small \begin{align*}
	A_{12}=\left[\begin{array}{c c c c c} 1&0&\cdots & 0 &0 \\1&1&\cdots & 0 &0 \\\vdots &\vdots &\ddots & \vdots &\vdots\\ 0&0&\cdots & 1 &0\\ 0&0&\cdots & 1 &1\\ \hline 0&0&\cdots & 0 &0 \end{array}   \right] ,
	\; A_{23}=\left[\begin{array}{c c c c c} 1&0&\cdots & 0 &0\\1&1&\cdots & 0 &0 \\\vdots &\vdots &\ddots & \vdots &\vdots\\ 0&0&\cdots & 1 &0 \\0&0&\cdots & 1 &1\\\hline  0&0&\cdots & 0 &1 \end{array}   \right], \;
	A_{31} =\left[\begin{array}{c c c c c | c c }  1&1&\cdots & 0 &0&0&0  \\0&1&\cdots & 0 &0 &0&0\\\vdots &\vdots &\ddots & \vdots &\vdots& \vdots &\vdots\\ 0&0&\cdots &1& 1 &0&0\\ 0&0&\cdots &0& 1 &1&0 \end{array}   \right].
	\end{align*}}
By using \eqref{determinantofBlockmatrix}, one has 
$$\det (M) =\det (A_{31}A_{12}A_{23}) =\det \bmt{3&1&0&\cdots & 0 &0&0 \\3&3&1&\cdots & 0 &0 &0\\ 1&3&3&\cdots & 0 &0 &0\\\vdots &\vdots&\vdots &\ddots & \vdots &\vdots& \vdots \\ 0&0&0&\cdots &3&1&0\\ 0&0&0&\cdots &3&3&1\\ 0&0&0&\cdots & 1 &3&3 } \neq 0,$$
since the last matrix is a Toeplitz matrix which is invertible by \cite[Lemma 3.4]{AN2018}.

\noindent{\bf \underline{Subcase 2: $\beta= a-b+1\geq 2$.}} More precisely, we consider the ideal
\begin{align*}
I=\begin{cases}
(x^a,y^b-x^{b-1}z,z^3,x^\beta y^{b-\beta},y^{b-\beta}z^2) &\text{if}\quad \gamma =1\\
(x^a,y^b-x^{b-2}z^2,z^3,x^{\beta+1} y^{b-\beta},y^{b-\beta}z) &\text{if}\quad \gamma =2.
\end{cases}
\end{align*}
In this case, one has $k=b-1$ and
$$H_{R/I}(b)=H_{R/I}(b-1)=\begin{cases}
3b-\beta-1&\text{if}\quad \gamma=1\\
3b-2\beta &\text{if}\quad \gamma=2
\end{cases}$$
by Lemma~\ref{lemmakeyforc=3}. Let $\mathcal{B}$ and $\mathcal{B}^\prime$ be a $K$-linear basis of $[R/I]_{b-1}$ and $[R/I]_{b}$, respectively. Set $L= x+y+z$.  By Proposition~\ref{Proposition2.5}, it is enough to show that
$$\times L: [R/I]_{b-1}\longrightarrow [R/I]_{b}$$
is an isomorphism, or equivalently the matrix representation $M$ of $\times L$ with respect to these bases  has nonzero determinant. To do it, we first consider the case where $\gamma=1.$  In this case, we write $\mathcal{B} =\mathcal{B}_1\sqcup \mathcal{B}_2\sqcup \BB_3,$
where 
\begin{align*}
\mathcal{B}_1 &= \{ u_{i} =x^{b-i} y^{i-1} \mid i=1,2,\ldots,b \}\\
\mathcal{B}_2&= \{ v_i =x^{b-1-i} y^{i-1}z \mid i=1,2,\ldots,b-1 \}\\
\mathcal{B}_3&= \{ w_i =x^{b-2-i} y^{i-1}z^2 \mid i=1,2,\ldots,b-\beta \}
\end{align*}
and  $\mathcal{B}^\prime =\mathcal{B}_1^\prime\sqcup \mathcal{B}_2^\prime\sqcup \BB_3^\prime,$
where 
\begin{align*}
\mathcal{B}_1^\prime  &= \{u_i^\prime \mid 1\leq i\leq b   \}\quad \text{where}\quad  u_i^\prime=\begin{cases}
x^{b+1-i} y^{i-1}&\text{if}\;i=1,2,\ldots,b-\beta \\  x^{b-i} y^{i}&\text{if}\; i=b-\beta+1,\ldots,b 
\end{cases}\\
\mathcal{B}_2^\prime &= \{  v_{i}^\prime  =x^{b-1-i} y^{i}z \mid i=1,2,\ldots,b-1 \}\\
\mathcal{B}_3^\prime &= \{ w_i^\prime  =x^{b-1-i} y^{i-1}z^2 \mid i=1,2,\ldots,b-\beta \}.
\end{align*}
Since $z^3=0$ in $R/I$, it is easy to see that 
\begin{align*}
\times L(u_i)&=x^{b+1-i}y^{i-1} +x^{b-i}y^{i} +x^{b-i}y^{i-1}z\\
\times L(v_i)&=x^{b-i}y^{i-1}z +x^{b-1-i}y^{i}z +x^{b-1-i}y^{i-1}z^2\\
\times L(w_i)&=x^{b-1-i}y^{i-1}z^2 +x^{b-2-i}y^{i}z^2.
\end{align*}
It follows that
$$M=\bmt{ A_{11} &\vdots & 0&\vdots & 0\\\cdots & \cdots &\cdots&\cdots&\cdots\\ A_{21} & \vdots & A_{22}& \vdots & 0\\\cdots & \cdots &\cdots&\cdots&\cdots\\ 0& \vdots & A_{32}& \vdots & A_{33}},$$
where \begin{align*}
A_{22} =\left[\begin{array}{c c c c c} 1&1&\cdots & 0 &0  \\0&1&\cdots & 0 &0 \\\vdots &\vdots &\ddots & \vdots  &\vdots\\ 0&0&\cdots & 1 &1\\ 0&0&\cdots & 0 &1 \end{array}   \right]\quad\text{and}\quad A_{33} =\left[\begin{array}{c c c c c} 1&0&\cdots & 0 &0  \\1&1&\cdots & 0 &0 \\\vdots &\vdots &\ddots & \vdots  &\vdots\\ 0&0&\cdots & 1 &0\\ 0&0&\cdots & 1 &1 \end{array}   \right].
\end{align*}
By \eqref{determinantofBlockmatrix}, we get
\begin{align*}
\det (M) =\det (A_{33}) \det (A_{22}) \det (A_{11}) =\det (A_{11})=1
\end{align*}
since
\begin{align*}
&A_{11}=\left[\begin{array}{c c c c c|c c c c c } 1&0&\cdots & 0 &0&0&0&\cdots & 0 &0 \\1&1&\cdots & 0 &0 &0&0&\cdots & 0 &0\\\vdots &\vdots &\ddots & \vdots &\vdots&\vdots &\vdots &\ddots & \vdots &\vdots \\ 0&0&\cdots & 1 &0 &0&0&\cdots & 0 &0\\0&0&\cdots & 1 &1 &0&0&\cdots & 0 &0\\ \hline 0&0&\cdots & 0 &0& 1&1 &\cdots & 0 &0\\  0&0&\cdots & 0 &0 &0&1&\cdots&0  &0\\\vdots &\vdots &\ddots & \vdots &\vdots&\vdots &\vdots &\ddots & \vdots &\vdots\\ 0&0&\cdots & 0 &0 &0&0&\cdots & 1 &1\\0&0&\cdots & 0 &0 &0&0&\cdots & 0 &1  \end{array}   \right].
\end{align*}
To complete this subcase, we consider the case where $\gamma=2$. In this case, we write $\mathcal{B} =\mathcal{B}_1\sqcup \mathcal{B}_2\sqcup \BB_3,$
where 
\begin{align*}
\mathcal{B}_1 &= \{ u_{i} =x^{b-i} y^{i-1} \mid 1\leq i\leq b \}\\
\mathcal{B}_2&= \{ v_i =x^{b-1-i} y^{i-1}z \mid  1\leq i\leq b-\beta \}\\
\mathcal{B}_3&= \{ w_i =x^{b-2-i} y^{i-1}z^2 \mid  1\leq i\leq b-\beta \}
\end{align*}
and  $\mathcal{B}^\prime =\mathcal{B}_1^\prime\sqcup \mathcal{B}_2^\prime\sqcup \BB_3^\prime,$
where 
\begin{align*}
\mathcal{B}_1^\prime  &= \{u_i^\prime \mid 1\leq i\leq b   \}\quad \text{where}\quad  u_i^\prime=\begin{cases}
x^{b+1-i} y^{i-1}&\text{if}\;i=1,2,\ldots,b-\beta \\  x^{b-i} y^{i}&\text{if}\; i=b-\beta+1,\ldots,b 
\end{cases}\\
\mathcal{B}_2^\prime &= \{  v_{i}^\prime  =x^{b-i} y^{i-1}z \mid i=1,2,\ldots,b-\beta \}\\
\mathcal{B}_3^\prime &= \{ w_i^\prime  =x^{b-1-i} y^{i-1}z^2 \mid i=1,2,\ldots,b-\beta \}.
\end{align*}
Since $z^3=0$ in $R/I$, it is easy to see that 
\begin{align*}
\times L(u_i)&=x^{b+1-i}y^{i-1} +x^{b-i}y^{i} +x^{b-i}y^{i-1}z\\
\times L(v_i)&=x^{b-i}y^{i-1}z +x^{b-1-i}y^{i}z +x^{b-1-i}y^{i-1}z^2\\
\times L(w_i)&=x^{b-1-i}y^{i-1}z^2 +x^{b-2-i}y^{i}z^2.
\end{align*}
Therefore 
$$M=\bmt{ A_{11} &\vdots & 0&\vdots & 0\\\cdots & \cdots &\cdots&\cdots&\cdots\\ A_{21} & \vdots & A_{22}& \vdots & 0\\\cdots & \cdots &\cdots&\cdots&\cdots\\ 0& \vdots & A_{32}& \vdots & A_{33}},$$
where $A_{11}, A_{22}$ and $A_{33}$ are the same forms as in the case $\gamma=1.$
By \eqref{determinantofBlockmatrix}, one has
\begin{align*}
\det (M) =\det (A_{33}) \det (A_{22}) \det (A_{11}) =1.
\end{align*}

\noindent{\bf \underline{Case 2: $\beta\geq  a-b+3$.}}
In this case,  one has $k=\frac{a+b-\beta-1}{2}$ and
$$H_{R/I}(k+1)=H_{R/I}(k)=\begin{cases}
2a+b-3\beta +2&\text{if}\quad \gamma=1\\
a+2b-3\beta+2 &\text{if}\quad \gamma=2
\end{cases}$$
by Lemma~\ref{lemmakeyforc=3}. Let $\mathcal{B}$ and $\mathcal{B}^\prime$ be a $K$-linear basis of $[R/I]_{k}$ and $[R/I]_{k+1}$, respectively. Set $L= x+y+z$.  By Proposition~\ref{Proposition2.5}, it is enough to show that
$$\times L: [R/I]_{k}\longrightarrow [R/I]_{k+1}$$
is an isomorphism, or equivalently the matrix representation $M$ of $\times L$ with respect to these bases  has nonzero determinant. To do it, we consider the case where $\gamma=1.$ The case $\gamma=2$ is similarly proved, even more simple. In the case, the ideal 
$$I=(x^a,y^b-x^{b-1}z,z^3,x^{a-b+1} y^{b-\beta},y^{b-\beta}z^2).$$
Clearly,  since $\beta \geq a-b+3, a+1-\beta\leq k\leq b-2\leq a-2$. Therefore, it is easy to show that a $K$-linear basis of $[R/I]_{k}$ is $\mathcal{B} =\mathcal{B}_1\sqcup \mathcal{B}_2\sqcup \mathcal{B}_3,$
where
\begin{align*}
\mathcal{B}_1 &= \{ u_i \mid 1\leq i\leq a-\beta +1 \}\\
\mathcal{B}_2 &= \{ v_i \mid 1\leq i\leq a-\beta +1 \}\\
\mathcal{B}_3 &= \{ w_i =x^{k-1-i} y^{i-1} z^2\mid i=1,2,\ldots, b-\beta  \}.
\end{align*}
where 
\begin{align*}
u_i&=\begin{cases}
x^{k+1-i} y^{i-1} &\text{if}\quad 1\leq i\leq b-\beta\\
x^{a-\beta+1-i} y^{k-a+\beta+i -1} &\text{if}\quad b-\beta +1\leq i\leq a-\beta+1
\end{cases}\\
v_i&=\begin{cases}
x^{k-i} y^{i-1}z &\text{if}\quad 1\leq i\leq b-\beta\\
x^{a-\beta+1-i} y^{k-a+\beta+i -2} z&\text{if}\quad b-\beta +1\leq i\leq a-\beta+1
\end{cases}
\end{align*}
and a $K$-linear basis of $[R/I]_{k+1}$ is $\mathcal{B}^\prime =\mathcal{B}_1^\prime\sqcup \mathcal{B}_2^\prime\sqcup \mathcal{B}_3^\prime,$
where
\begin{align*}
\mathcal{B}_1^\prime &= \{ u_i^\prime \mid 1\leq i\leq a-\beta +1 \}\\
\mathcal{B}_2^\prime &= \{ v_i^\prime \mid 1\leq i\leq a-\beta +1 \}\\
\mathcal{B}_3^\prime &= \{ w_i^\prime =x^{k-i} y^{i-1} z^2\mid i=1,2,\ldots, b-\beta  \}.
\end{align*}
where 
\begin{align*}
u_i^\prime&=\begin{cases}
x^{k+2-i} y^{i-1} &\text{if}\quad 1\leq i\leq b-\beta\\
x^{a-\beta+1-i} y^{k-a+\beta+i} &\text{if}\quad b-\beta +1\leq i\leq a-\beta+1
\end{cases}\\
v_i^\prime&=\begin{cases}
x^{k+1-i} y^{i-1}z &\text{if}\quad 1\leq i\leq b-\beta\\
x^{a-\beta+1-i} y^{k-a+\beta+i -1} z&\text{if}\quad b-\beta +1\leq i\leq a-\beta+1.
\end{cases}
\end{align*}
It follows that 
\begin{align*}
\times L(u_i)&=\begin{cases}
x^{k+2-i}y^{i-1} +x^{k+1-i}y^{i} +x^{k+1-i}y^{i-1}z&\text{if}\quad 1\leq i\leq b-\beta\\
x^{a-\beta+2-i} y^{k-a+\beta+i -1} +x^{a-\beta+1-i} y^{k-a+\beta+i } \\
+x^{a-\beta+1-i} y^{k-a+\beta+i -1}z&\text{if}\quad b-\beta+1\leq i\leq a-\beta+1
\end{cases}\\
\times L(v_i)&=\begin{cases}
x^{k+1-i} y^{i-1}z +x^{k-i} y^{i}z +x^{k-i} y^{i-1}z^2&\text{if}\quad 1\leq i\leq b-\beta\\
x^{a-\beta+2-i} y^{k-a+\beta+i -1} z +x^{a-\beta+1-i} y^{k-a+\beta+i} z \\
+x^{a-\beta+1-i} y^{k-a+\beta+i -1} z^2&\text{if}\quad b-\beta+1\leq i\leq a-\beta+1
\end{cases}\\
\times L(w_i)&=x^{k-i} y^{i-1} z^2 +x^{k-1-i} y^{i} z^2
\end{align*}
and hence
$$M=\bmt{ A_{11} &\vdots & 0&\vdots & 0\\\cdots & \cdots &\cdots&\cdots&\cdots\\ A_{21} & \vdots & A_{22}& \vdots & 0\\\cdots & \cdots &\cdots&\cdots&\cdots\\ 0& \vdots & A_{32}& \vdots & A_{33}}.$$
A simple computation shows that $A_{11}, A_{22}$ and  $A_{33}$ are the invertible matrices. It follows from \eqref{determinantofBlockmatrix} that
\begin{align*}
\det (M) =\det (A_{33}) \det (A_{22}) \det (A_{11})\neq 0.
\end{align*}

\end{proof}

Our second main result is the following.
\begin{Theo}\label{Theorem3.13}
Consider the ideal $I$ as in \eqref{IdealGorenstein}. If one of the $a,b$ and $c$ is equal to 3, then $R/I$ has the WLP.
\end{Theo}

\begin{proof}
If $b=3$, then $R/I$ has the WLP by Proposition~\ref{Propositionforb=3}. By symmetry of $x$ and $z$, we can always assume that $a\geq c.$ Therefore, it suffices to prove that $R/I$ has the WLP whenever $c=3.$ It follows from Propositions~\ref{Propositionforbetasmallandc=3} and~\ref{Proposition3.13} that $R/I$ has the WLP in this case.
\end{proof}

\section*{Acknowledgments}
The authors thank the referee for a careful reading and useful comments that improved the presentation of the article. 
The first author was partially supported by the grant MTM2016-78623-P. The second author was partially supported by the project ``\`Algebra i Geometria Algebraica" under grant number 2017SGR00585 and  by Vietnam National Foundation for Science and Technology Development (NAFOSTED) under grant number 101.04-2019.07.

\bibliographystyle{plain} 
\bibliography{WLP_Gorenstein3} 


\end{document}